\documentclass{amsart}
\usepackage{amsmath,amssymb,graphicx,tikz}
\usepackage{placeins}

\newcommand{\ab}{\bar{\alpha}}
\newcommand{\bb}{\bar{\beta}}
\newcommand{\C}{\mathcal C}
\renewcommand{\O}{\mathcal O}
\newcommand{\F}{{\mathbf F}}
\newcommand{\Fp}{\F_p}
\newcommand{\Fq}{\F_q}
\newcommand{\Fpbar}{\bar{\F}_p}
\newcommand{\Fqbar}{\bar{\F}_q}

\newcommand{\GL}{\mathrm{GL}}
\newcommand{\Z}{\mathbf Z}
\newcommand{\Q}{\mathbf Q}
\newcommand{\Zp}{\Z_p}
\newcommand{\Qp}{\Q_p}
\newcommand{\Zpn}{\Z/p^n}
\newcommand{\Zpns}{(\Zpn)^\times}
\newcommand{\LCM}{\mathrm{LCM}}
\newcommand{\oa}{o_\alpha}

\newcommand{\oab}{o_{\ab}}
\newcommand{\obb}{o_{\bb}}
\newcommand{\oap}{\oa'}
\newcommand{\oabp}{\oab'}

\renewcommand{\P}{\mathbf P}
\newcommand{\lt}{\tilde{\lambda}}
\newcommand{\N}{{\mathbf N}}
\newcommand{\Qbar}{\bar{\Q}}
\newcommand{\ZN}{\Z/N}
\newcommand{\ZNs}{(\Z/N)^\times}
\newcommand{\vep}{\varepsilon}
\renewcommand{\oe}{o_\vep}

\newcommand{\QQ}{\mathfrak Q}
\newcommand{\q}{\mathfrak q}

\newcommand{\Ql}{\Q_\ell}
\newcommand{\Zl}{\Z_\ell}

\newcommand{\voa}{v_{\alpha}}
\newcommand{\vob}{v_{\beta}}

\newcommand{\rkm}{r_k^\mu}
\newcommand{\skm}{s_k^\mu}
\newcommand{\tkm}{t_k^\mu}
\newcommand{\tom}{t_{\oab}^\mu}
\newcommand{\zkm}{z^\mu}

\newcommand{\inj}{\hookrightarrow}
\newcommand{\tors}{\mathrm{tors}}
\newcommand{\p}{\mathfrak p}

\newcommand{\dcp}{\mathrm{DCT}}
\newcommand{\dcpm}{\mathrm{DCT}^{\text{mult}}}
\newcommand{\dcpo}{\mathrm{DCT}^{\text{ord}}}
\DeclareMathOperator{\Gal}{Gal}
\newcommand{\Frob}{\mathrm{Frob}}
\newcommand{\St}{\mathrm{St}}
\DeclareMathOperator{\End}{End}
\DeclareMathOperator{\disc}{disc}

\newcommand{\mat}[4]{\left( \begin{array}{cc} #1 & #2 \\ #3 & #4 \end{array} \right)}
\renewcommand{\vec}[2]{\left( \begin{array}{c} #1 \\ #2 \end{array} \right)}

\newtheorem{thm}{Theorem}
\numberwithin{thm}{section}
\newtheorem{lemma}[thm]{Lemma}
\newtheorem{prop}[thm]{Proposition}
\newtheorem{cor}[thm]{Corollary}

\theoremstyle{definition}

\newtheorem{rmk}[thm]{Remark}
\newtheorem{example}[thm]{Example}

\title{Dedekind zeta functions of non-Galois torsion fields of elliptic curves}
\author{Robert Pollack and Tom Weston}
\dedicatory{To Lucky the yellow labrador (2010-2026), who did everything in his power to prevent this paper from being completed}

\begin{document}

\maketitle

\section{Introduction}

Let $E$ be an elliptic curve over $\Q$.  Fix a point $P \in E(\Qbar)$ of odd order $N$
and let $K$ be the number field generated by the coordinates of $P$.  Under the assumption that the full
$N$-torsion field $\Q(E[N])$ has maximal Galois group, we compute the factorization of primes in the ring of integers $\O_K$.  We first give all possible factorization types (i.e., how many primes of which residual degrees occur in the factorization) for unramified primes and determine how often each occurs; see
Section~\ref{sec:dist}.  This depends only on $N$ and not on $E$.

When $E$ is semistable and has good ordinary reduction (without a companion form) at all $p$ dividing $N$, we consider also ramified primes and give an algorithm to count the number of ideals of $\O_K$ of any norm;
see Section~\ref{s:dzf}.  This allows one to compute coefficients of the Dedekind zeta function of $K$.  For example, when
$E=X_0(11)$ and $N=63$, then $K$ is a number field of degree $3456 $ and the first 60 nonzero coefficients of $\zeta_K(s) = \sum_{n=1}^{\infty} c_n n^{-s}$ are
given in Table~\ref{t:dzf}.

\begin{table} 
\begin{tabular}{|c|c||c|c||c|c||c|c|} \hline
$n$ & $c_n$ & $n$ & $c_n$ & $n$ & $c_n$ & $n$ & $c_n$ \\ \hline
1	&	1	&	1331	&	66552	&	5929	&	453492	&	9289	&	648	\\
7	&	18	&	1607	&	36	&	6241	&	18	&	9317	&	1197936	\\
11	&	72	&	2191	&	648	&	6269	&	36	&	9349	&	108	\\
49	&	171	&	2203	&	36	&	6311	&	36	&	9371	&	36	\\
77	&	1296	&	2401	&	6309	&	6337	&	36	&	10067	&	36	\\
121	&	2652	&	3167	&	36	&	6359	&	36	&	10069	&	36	\\
313	&	36	&	3181	&	36	&	6561	&	18	&	10813	&	2592	\\
343	&	1158	&	3443	&	2592	&	6677	&	2592	&	10957	&	36	\\
539	&	12312	&	3773	&	83376	&	6743	&	2592	&	11071	&	36	\\
607	&	36	&	3853	&	36	&	6881	&	648	&	11249	&	648	\\
613	&	36	&	4249	&	648	&	6959	&	36	&	11299	&	108	\\
847	&	47736	&	4291	&	648	&	7517	&	36	&	11437	&	36	\\
983	&	36	&	4643	&	36	&	7723	&	36	&	11449	&	216	\\
1217	&	36	&	4651	&	36	&	8519	&	648	&	11897	&	36	\\
1327	&	36	&	4801	&	108	&	9283	&	36	&	12491	&	36	\\ \hline
\end{tabular}
\vspace{0.5cm}
\caption{Coefficients of $\zeta_K(s)$ for $X_0(11)$ and $N=63$} \label{t:dzf}
\end{table}

A second application of this work is to analyze basic properties of a family of homomorphisms
$$E(\Fp) \to E(\Fqbar)$$
which we define for primes of good reduction $p,q$ such that $p$ is not anomalous.  Under the assumption that $E(\Fp)$ is cyclic of odd order, our results allow
us to determine, among other things, the field of definition of the image of $E(\Fp)$ in $E(\Fqbar)$; that is, the (various) minimal $d$ such that the image of 
$E(\Fp)$ lies in $E(\F_{q^d})$.

Our approach relies on Proposition~\ref{prop:pdc} relating prime factorizations in non-Galois extensions to double cosets.  The bulk of the paper is a determination of the double coset decompositions
$$\Gamma \backslash \Gal(\Q(E[N])/\Q) / D$$
as $D$ various over all possible cyclic subgroups; here $\Gamma$ is the stabilizer of a single point $P$ of order $N$.
The complete answer to this question in the unramified case requires the theory of arithmetic functions as well as consideration of the Newton polygons of a certain family of determinant polynomials which arise in the calculations.  We then consider certain ramified cases in order to enable the calculation of the Dedekind zeta function.

We would like to thank Santiago Arango-Piñeros, Alina Bilialova and Farshid Hajir for helpful conversations related to this work.

\section{Prime decomposition in non-Galois extensions}

\FloatBarrier

We begin with a result on prime ideals in non-Galois extensions which, while known (but perhaps not \emph{well} known) we were unable to find in the literature in this form; see \cite{wood} for a brief discussion of this result from a different point of view as well as lamentations on the state of the literature.  (At the conclusion of the preparation of this paper we became aware that this result is proved in \cite[Section 4, Lemma 4]{HajirWong}.  We include our proof as well as it gives a somewhat different perspective.)

Let $L/E$ be a Galois extension of number fields and let $K$ denote an intermediate field.
Set $G = \Gal(L/E)$ and $\Gamma=\Gal(L/K)$.  Fix a prime ideal $\QQ$ of $L$ lying over a prime $\q$ of $E$.  Let $D$ (resp.\ $I$) denote the decomposition (resp.\ inertia) group of $\QQ$ in $G$.  We write $e(\QQ/\q)$ and $f(\QQ/\q)$ for the ramification degree and residual degree, with the obvious corresponding notation to represent the same quantities for primes in other extensions.

\begin{table}
\centerline{
\begin{tikzpicture}

    \node (Q1) at (0,0) {$E$};
    \node (Q2) at (0,1) {$K$};
    \node (Q3) at (0,2) {$L$};
    \node (Q4) at (1,0) {$\q$};
    \node (Q5) at (1,2) {$\QQ$};

    \draw (Q1)--(Q2);
    \draw (Q2) to node[right] {$\Gamma$} (Q3);
    \draw (Q1) -- (-0.5,0) to node[left]{$I \subseteq D \subseteq G$} (-0.5,2) -- (Q3);

    \end{tikzpicture}}
\caption{Set up of Proposition~\ref{prop:pdc}} \label{f:field1}
\end{table}

\FloatBarrier

\begin{prop} \label{prop:pdc}
There is a bijection
\begin{align} \label{eq:dcgal}
\Gamma \backslash G / D &\to \{ \text{primes of $K$ over $\q$} \} \\
\Gamma g D &\mapsto g \QQ \cap K. \nonumber
\end{align}
For any $g \in G$, we have
\begin{align*}
e(g \QQ \cap K/\q) &= \frac{\# \Gamma g I}{\# \Gamma} \\
e(g \QQ \cap K/\q)f(g \QQ \cap K/\q) &= \frac{\# \Gamma g D}{\# \Gamma}
\end{align*}
\end{prop}

\FloatBarrier

\begin{proof}
If $g' \in \Gamma g D$, then we can write $g' = \gamma g d$ with $\gamma \in \Gamma$
and $d \in D$.  Thus
\begin{align*}
g'\QQ \cap K &= \gamma g d \QQ \cap K \\
&= \gamma g\QQ \cap K \\
&= \gamma(g \QQ \cap K) \\
&= g \QQ \cap K
\end{align*}
since $\gamma$ fixes $K$.  Thus the map (\ref{eq:dcgal}) is well defined.
The surjectivity is clear since $G$ acts transitively on the primes over
$\q$.  For injectivity, suppose that $g\QQ \cap K = g' \QQ \cap K$.  Since $L/K$ is Galois, 
there exists
$\gamma \in \Gamma$ such that $g \QQ =\gamma g' \QQ$.  Thus $g^{-1}\gamma g' \in D$ and
$$g' = \gamma^{-1} g (g^{-1} \gamma g')$$
shows that $g' \in \Gamma g D$.

For the two numerical formulas,
fix $g \in G$ and let $D' = gDg^{-1}$ be the decomposition group of $g\QQ$.  For the remainder of the proof, for any intermediate extension $E \subseteq M_1 \subseteq M_2 \subseteq L$, we define
$$ef(M_2/M_1) = e(g\QQ \cap M_2/g\QQ \cap M_1)f(g\QQ \cap M_2/g\QQ \cap M_1).$$

\begin{table}
\centerline{
\begin{tikzpicture}

    \node (Q1) at (0,0) {$E$};
    \node (Q2) at (0,1) {$L^{D'} \cap K$};
    \node (Q3) at (-2,2) {$L^{D'}$};
    \node (Q4) at (2,2) {$K$};
    \node (Q5) at (0,3) {$L^{D'} K$};
    \node (Q6) at (0,5) {$L$};

    \draw (Q1)--(Q2);
    \draw (Q2)--(Q3);
    \draw (Q2)--(Q4);
    \draw (Q3)--(Q5);
    \draw (Q4)--(Q5);
    \draw (Q5) to node[right] {$D' \cap \Gamma$} (Q6); 
    \draw (Q6) to node[right] {$\Gamma$} (Q4);
    \draw (Q6) to node[left] {$D'$} (Q3);

    \end{tikzpicture}}
\caption{Intermediate fields $L/E$} \label{f:field2}
\end{table}

Consider the field diagram in Table~\ref{f:field2}.
The prime $\q$ is totally split in the decomposition field $L^{D'}$.  In particular, we must have
$ef(L^{D'} \cap K/E)=1$, so that
$$ef(K/E) = ef(K/L^{D'} \cap K).$$
As also
$$ef(L^{D'}K/K) \leq ef(L^{D'}/L^{D'} \cap K) = 1$$
and $ef$ is multiplicative in towers, we find that
$$ef(K/E) = ef(L^{D'}K/L^{D'}).$$
But $ef(L/L^{D'})=\#D'$, so that finally
\begin{align*}
ef(K/E) &= \frac{\#D'}{\#D' \cap \Gamma} \\
&= [D' : D' \cap \Gamma] \\
&= [D : D \cap g^{-1}\Gamma g] \\
&= \frac{\# \Gamma g D}{\# \Gamma}
\end{align*}
by a standard formula for orders of double cosets.  A similar argument with $I$ replacing $D$ gives the formula purely for inertia.
\end{proof}

\section{Double coset types} \label{s:dcp}

Fix a finite group $G$ and a subgroup $\Gamma$.  Let $D$ denote another subgroup of $G$; for the purpose of this discussion we imagine $\Gamma$ being fixed and $D$ varying.  We say that $D$ has {\it double coset type}
$$a_1 \times b_1 + \cdots + a_m \times b_m$$
if $G$ decomposes as
a disjoint union of precisely $a_i$ double cosets $D g \Gamma$ of each order $b_i\#\Gamma$ and no other double cosets.  (We do for convenience allow repeated $b_i$, in which case we mean that the number of cosets of order $b_i \#\Gamma$ is the sum of the corresponding coefficients $a_i$.)

We will also have need of the following augmented version.  Let $I \subseteq D$ be subgroups.  
We say that the pair $(D,I)$ has {\it double coset type}
$$a_1 \times (b_1,c_1) + \cdots + a_m \times (b_m,c_m)$$
if there are $a_i$ double cosets $D g \Gamma$ of order $b_i\#\Gamma$
and each of these double cosets is a union of $\frac{b_i}{c_i}$ double cosets $I g \Gamma$ of order $c_i \#\Gamma$.
Note that double coset types are the specialization of double coset types of a pair in the case that $I=\{1\}$.

\begin{lemma} \label{lemma:factor}
In the setting of Proposition~\ref{prop:pdc}, let
$$a_1 \times (b_1,c_1) + \cdots + a_m \times (b_m,c_m)$$
be the double coset type of the pair $(D,I)$.  Then
the prime $\q$ factors in $\O_L$ as a product of $a_1+\cdots+a_m$ primes, of which $a_i$ have ramification index $c_i$ and residual degree $\frac{b_i}{c_i}$.
\end{lemma}
\begin{proof}
This is simply a restatement of Proposition~\ref{prop:pdc} in the terminology of double coset types.  Note that there is a cardinality preserving bijection (given by inversion) between double cosets
$\Gamma \backslash G / D$ and $D \backslash G / \Gamma$ so that the opposite order in
Proposition~\ref{prop:pdc} does not affect the numerics.
\end{proof}

The next lemma is not at all surprising but is very useful for limiting the collection of subgroups $D$ which we must consider.

\begin{lemma} \label{l:myst}
If $D$ and $D'$ are conjugate subgroups of $G$, then $D$ and $D'$ have the same double coset type.
\end{lemma}
\begin{proof}
Fix $g \in G$ such that $D'=gDg^{-1}$.  Then left multiplication by $g$ gives a bijection from double cosets of $D$ and $\Gamma$ to double cosets of $D'$ and $\Gamma$.  Indeed,
if $$C=\{ dc\gamma \hspace{0.5em} ; \hspace{0.5em} d \in D, \gamma \in \Gamma\}$$
is a double coset for $D$ containing $c \in G$, then 
$$gC = \{ gdg^{-1}gc\gamma \hspace{0.5em}; \hspace{0.5em} gdg^{-1} \in D', \gamma \in \Gamma\}$$
is a double coset for $D'$ containing $gc$.  Since multiplication by $g$ is itself bijective, the lemma follows.
\end{proof}

We will also need the following result on double coset types of products.  Let $\Gamma_1 \subseteq G_1$ and $\Gamma_2 \subseteq G_2$ be finite groups.
Let $I \subseteq D$ be subgroups of $G_1 \times G_2$.  Let $I_i \subseteq D_i$ be the projections of $I \subseteq D$ to $G_i$.  Note that we are not assuming that $D \cong D_1 \times D_2$.  We wish to relate double coset types for $D \backslash G_1 \times G_2 / \Gamma_1 \times \Gamma_2$ to those for
$D_1 \backslash G_1 / \Gamma_1$ and $D_2 \backslash G_2 /\Gamma_2$.

We define a product on double coset types by extending the formulas
$$(a \times b)_1 \otimes (a' \times b')_2 = \frac{aba'b'}{\LCM(b,b')} \times \LCM(b,b')_{1 \times 2}$$
$$\bigl(a \times (b,c)\bigr)_1 \otimes \bigl(a' \times (b',c')\bigr)_2 = \frac{aba'b'}{\LCM(b,b')} \times \bigl(\LCM(b,b'),\LCM(c,c')\bigr)_{1 \times 2}$$
by linearity.  Here the subscripts indicate which group $G_1$, $G_2$, $G_1 \times G_2$ the double coset types reference.

\begin{lemma} \label{lemma:dcpprod}
With the above notation, if the double coset type of $D_i$ is $(\sum a_{i,j} \times b_{i,j})_i$, then the double coset type of
$D$ is $\left( \sum a_{1,j} \times b_{1,j}\right)_1 \otimes \left(\sum a_{2,j} \times b_{2,j} \right)_2$.
The analogous formula holds for double coset types of pairs.
\end{lemma}
\begin{proof}
Let $W_i$ denote the coset space $G_i/\Gamma_i$.  A double coset $D_i g_i \Gamma_i$ of order $b_i \#\Gamma_i$
corresponds to an orbit $O_i$ of order $b_i$ for the left action of $D_i$ on $W_i$.  The orbit of $D$ acting on $W_1 \times W_2$
is then of order $\LCM(b_1,b_2)$, so that the double coset $(\Gamma_1 \times \Gamma_2) (g_1,g_2) D$ has order $\LCM(b_1,b_2)\#\Gamma_1\#\Gamma_2$.  That is,
$O_1 \times O_2$ partitions into $\frac{b_1b_2}{\LCM(b_1,b_2)}$ orbits for the action of $D$, each of order $\LCM(b_1,b_2)$.
This shows that the product of $(1 \times b_1)_1$ and $(1 \times b_2)_2$ is indeed
$$\frac{b_1 b_2}{\LCM(b_1,b_2)} \times \LCM(b_1,b_2)_{1 \times 2}$$
as claimed.  The lemma follows easily from this.
\end{proof}

\section{Conjugacy classes modulo prime powers} \label{sec:cc}

As $\GL_2(\Z/N)$ decomposes as a product of the groups $\GL_2(\Z/p^n)$ as $p^n$ ranges over prime power factors of $N$, 
by Lemma~\ref{lemma:dcpprod} it is not difficult to recover the general case of double coset types for $\GL_2(\Z/N)$ from the case where $N$ is a prime power.
By Lemma~\ref{l:myst} it will suffice to consider subgroups $D \subseteq \GL_2(\Zpn)$ up to conjugacy.  In this section we therefore determine the conjugacy classes in $\GL_2(\Z/p^n)$.

We begin with an elegant characterization of conjugacy in $\GL_2(\Zp)$; we learned it from
\cite{cent}.
For a matrix $A \in \GL_2(\Zp)$, write $\chi_A \in \Zp[x]$ for its 
characteristic polynomial.  Let
$\mu(A)$ denote the largest  $\mu \in \Z \cup \{\infty\}$ such that $A$ is a scalar matrix modulo $p^{\mu}$.

\begin{lemma} \label{l:conjZp}
Two matrices $A,B \in \GL_2(\Zp)$ are conjugate if and only if $\chi_A = \chi_B$ and $\mu(A) = \mu(B)$.
\end{lemma}
\begin{proof}
It is clear that $\chi_A$ and $\mu(A)$ are conjugacy class invariants; we must show that they also determine the conjugacy class of $A$.
Clearly $A$ is scalar if and only if $\mu(A) = \infty$, in which case $\chi_A$ determines $A$ and the lemma is clear.  

We assume therefore that $A$ is not scalar, so that
$\chi_A$ is the minimal polynomial of $A$.  Thus $\Zp[A] \cong \Zp[x]/\chi_A$ and, in particular, is a free $\Zp$-module of rank 2; here we regard $\Zp[A]$ as a subring of the ring $M_2(\Zp)$ of $2 \times 2$ matrices with entries in $\Zp$.
The ring $\Zp[A]$ is a subring of $R_A := \Qp[A] \cap M_2(\Zp)$.  Note that $R_A$ is also a free $\Zp$-module of rank 2, as it is a compact open $\Zp$-submodule of
the two dimensional $\Qp$-vector space $\Qp[A]$.  It follows that $\Zp[A]$ has finite index $p^m$ in $R_A$.

Let $1$ and $\frac{a+bA}{p^m}$ be a $\Zp$-basis of $R_A$; as $p^m = [R_A:\Zp[A]]$, at least one of $a,b$ is a unit.   In fact,
$b$ must be a unit, for otherwise $$p^{m-1}\frac{a+bA}{p^m} - \frac{b}{p}A = \frac{a}{p}$$ would lie in $R_A \subseteq M_2(\Zp)$, contradicting $a$ being a unit.  As
the $2 \times 2$ matrix $a+bA$ is divisible by $p^m$, the action of $A$ on $(\Zp/p^m)^2$ must be by the scalar $-ab^{-1}$, which shows that
$m \leq \mu(A)$.  

By the definition of $\mu(A)$, we may choose $\alpha \in \Zp$ such that $A-\alpha$ is divisible by $p^{\mu(A)}$ in $M_2(\Zp)$.  Since
$A' := \frac{A-\alpha}{p^{\mu(A)}}$ commutes with $A$, it must lie in $\Qp[A]$ and thus in $R_A$.  It follows that $m \geq \mu(A)$.  We conclude that $m=\mu(A)$ and therefore that $1,A'$ is a $\Zp$-basis of $R_A$.

As $A$ is not scalar modulo $p^{\mu(A)+1}$, we must have that $A'$ is not a scalar on $(\Z_p/p)^2$.  We may therefore choose a basis $e_1,e_2$ of $\Zp^2$ such that $A'$ does not preserve the line
in $(\Zp/p)^2$ spanned by $e_1$.  Thus $e_1,A'(e_1)$ is a basis of $\Zp^2$ as well.  The matrix of $A'$ with respect to this basis is 
$$\mat{0}{-\tau}{1}{-\sigma}$$
where $A'$ has characteristic polynomial $x^2+\sigma x+\tau$.  As the characteristic polynomial of $A'$ is
$$p^{-2\mu(A)}\chi_A\left(p^{\mu(A)}x-\alpha\right)$$
and $\alpha$ is determined modulo $p^{\mu(A)}$ as the unique root of $\chi_A$ modulo $p^{\mu(A)}$, it follows that the conjugacy class of $A$ is entirely determined by $\chi_A$ and $\mu(A)$, as claimed.
\end{proof}

\begin{lemma} \label{l:muineq} For any $A \in \GL_2(\Zp)$, 
 $p^{2\mu(A)}$ divides the discriminant of $\chi_A$.
\end{lemma}
\begin{proof}
Set $\mu = \mu(A)$; we may assume that $\mu > 0$ or else there is nothing to prove.  By the definition of $\mu$ we may choose
$\alpha \in \Zp$ such that we can write
$$A = \mat{\alpha +ap^\mu}{bp^\mu}{cp^\mu}{\alpha+dp^\mu }$$
for $a,b,c,d \in \Zp$.
Then the discriminant of $\chi_A$ is given by
$$(\alpha+ap^\mu + \alpha +dp^\mu)^2 - 4\bigl( (\alpha+ap^\mu)(\alpha+dp^\mu)-bp^\mu cp^\mu) \bigr) =
p^{2\mu}(ad-bc).$$
\end{proof}

\begin{cor} \label{c:irregform} Let $A \in \GL_2(\Zp)$ satisfy $\mu(A) > 0$.  Then there are $\alpha, \beta \in \Zp^\times$ and $\nu \geq \mu(A)$ such that $A$ is conjugate to
$$\mat{ \alpha}{ \beta p^\nu}{ p^{\mu(A)}}{ \alpha}.$$
\end{cor}
\begin{proof}
Write $\chi_A = x^2-\sigma x + \tau$.  Set $\alpha = \frac{1}{2}\sigma$.  By Lemma~\ref{l:muineq},
$$v(\sigma^2-4\tau) \geq 2\mu(A)$$
so setting $\nu=v(\sigma^2-4\tau)-\mu(A)$ and $\beta = \frac{\sigma^2-4\tau}{p^{\mu(A)+\nu}}$ we have $\nu \geq \mu(A)$ and $\beta \in \Zp^\times$.  With these choices the given
matrix has the same characteristic polynomial and $\mu$ as $A$ does, so by Lemma~\ref{l:conjZp} they are conjugate.
\end{proof}

\begin{prop}   \label{p:ccpn}
The conjugacy classes in $GL_2(\Z/p^n)$ are those given in Table~\ref{t:cc}.
\end{prop}

\begin{table} 
\renewcommand{\arraystretch}{1}
$$\begin{array}{|l|l|l|l|} \hline
\text{Label} & \text{Representative} & \text{Parameters} &  \text{Size of each conjugacy class} \\ \hline 
I(\alpha) & \mat{\alpha}{0}{0}{\alpha} & \begin{array}{l} \alpha \in \Zpns \end{array} &  1 \\  \hline
I'_{\mu,\nu}(\alpha,\beta) & \mat{\alpha}{\beta p^\nu}{p^\mu}{\alpha} & \begin{array}{l} \alpha \in \Zpns \\ 1 \leq \mu < \nu < n \\ \beta \in (\Z/p^{n-\nu})^\times \end{array}&  (p^2-1)p^{2(n-\mu)-2} \\ \hline
I'_{\mu}(\alpha) & \mat{\alpha}{0}{p^\mu}{\alpha} & \begin{array}{l} \alpha \in \Zpns \\ 1 \leq \mu < n \end{array}& (p^2-1)p^{2(n-\mu)-2} \\ \hline
I^-_{\mu}(\alpha,\beta) & \mat{\alpha}{\beta p^\mu}{p^\mu}{\alpha} & \begin{array}{l} \alpha \in \Zpns \\ 1 \leq \mu < n \\ \beta \in (\Z/p^{n-\mu})^\times \\ \beta \notin (\Z/p^{n-\mu})^{\times 2} \end{array} &(p-1)p^{2(n-\mu)-1} \\  \hline
I^+_{\mu}(\alpha,\beta) & \mat{\alpha}{\beta p^\mu}{p^\mu}{\alpha} & \begin{array}{l} \alpha \in \Zpns \\ 1 \leq \mu < n \\ \beta \in (\Z/p^{n-\mu})^{\times 2}\end{array}&(p+1)p^{2(n-\mu)-1} \\ \hline
\text{II}(\alpha,\beta) & \mat{\alpha}{\beta p}{1}{\alpha} & \begin{array}{l} \alpha \in \Zpns \\ \beta \in \Z/p^{n-1} \end{array} & (p^2-1)p^{2n-2} \\ \hline
\text{III}(\alpha,\beta) & \mat{\alpha}{0}{0}{\beta} & \begin{array}{l} \alpha,\beta \in \Zpns \\ \alpha \not\equiv \beta \bmod{p} \end{array} & (p+1)p^{2n-1} \\ \hline 
\text{IV}(\alpha,\beta) & \mat{0}{\alpha}{1}{\beta} & \begin{array}{l} \alpha \in \Zpns \\ \beta \in \Zpn \\ \beta^2 +4\alpha \in \Zpns \\ \beta^2+4\alpha \notin (\Zpn)^{\times 2}  \end{array} & (p-1)p^{2n-1} \\ \hline
\end{array}$$
\renewcommand{\arraystretch}{1}
\caption{Conjugacy Classes in $\GL_2(\Z/p^n)$} \label{t:cc}
\end{table}

\begin{proof}
By \cite[Theorem III.2]{McDonald}, every regular element of $G$ can be put into rational canonical form and thus its conjugacy class is determined entirely by its characteristic polynomial.   (Here $g$ is {\it regular} if the reduction of its minimal polynomial equals the minimal polynomial of its reduction.)  For
$\GL_2(\Zpn)$, a scalar matrix is always regular, while a non-scalar matrix is regular if and only if it is not congruent to a scalar matrix modulo $p$: these are precisely cases $\text{I}$, $\text{II}$, $\text{III}$ and $\text{IV}$.  Cases $\text{I}$ and $\text{II}$ cover characteristic polynomials with
non-unit discriminant (zero and non-zero, respectively) while cases $\text{III}$ and $\text{IV}$ cover characteristic polynomials (reducible and irreducible, respectively) with unit discriminant.

It remains to consider the irregular case.
Let $g \in \GL_2(\Zpn)$ be irregular.  Choosing an arbitrary lift of $g$ to $\GL_2(\Zp)$ and applying 
Corollary~\ref{c:irregform}, we know that $g$ is conjugate (over $\Zp$ and thus over $\Zpn$ as well) to a matrix of the form
$$\mat{ \alpha}{ \beta p^\nu}{ p^{\mu}}{ \alpha}$$
for $\alpha,\beta \in \Zpns$ and $1 \leq \mu \leq \nu \leq n$.  All that remains to be shown is that none of these matrices are conjugate to one another and to verify the sizes of the conjugacy classes.  In fact, it suffices to do the latter: a sufficiently invested reader can verify that as all parameters vary, the classes $I(\alpha)$, $I'_{\mu,\nu}(\alpha,\beta)$, $I'_\mu(\alpha)$, $I^-_\mu(\alpha,\beta)$ and $I^+_\mu(\alpha,\beta)$ contain a total of $(p-1)p^{4n-4}$ elements, which is precisely the number of lifts of scalar elements in $\GL_2(\Fp)$ to $\GL_2(\Zpn)$.

Computing the size of each conjugacy class is a straightforward centralizer calculation.  We illustrate in the most interesting case of $I^+_\mu(\alpha,\beta)$.  We must find all
$$\mat{a}{b}{c}{d} \in \GL_2(\Zpn)$$
such that
\begin{align*}
\mat{\alpha}{\beta p^\mu}{p^\mu}{\alpha} \mat{a}{b}{c}{d} &=  \mat{a}{b}{c}{d} \mat{\alpha}{\beta p^\mu}{p^\mu}{\alpha} \\
\mat{\alpha a + \beta p^\mu c}{\alpha b + \beta p^\mu d}{p^\mu a + \alpha c}{p^\mu b + \alpha d} &= \mat{a \alpha + b p^\mu}{a\beta p^\mu+b\alpha}{c \alpha+d p^\mu}{c\beta p^\mu+d\alpha}.
\end{align*}
Making the obvious cancellations, this gives rise to the system of equations
\begin{align*}
\beta p^\mu c &= bp^\mu \\
\beta p^\mu d &= a\beta p^\mu \\
p^\mu a &= dp^\mu \\
p^\mu b &= c\beta p^\mu
\end{align*}
Clearly these are equivalent to
\begin{align*}
d &\equiv a \pmod{p^{n-\mu}} \\
b &\equiv \beta c \pmod{p^{n-\mu}}
\end{align*}
Thus an element of the centralizer has the form
$$\mat{a}{\beta c + xp^{n-\mu}}{c}{a+yp^{n-\mu}}$$
for $x,y \in \Z/p^{\mu}$.
However, such a matrix need not be invertible: it has determinant
$$a(a+yp^{n-\mu}) - c(\beta c + xp^{n-\mu}) = a^2-\beta c^2 + p^{n-\mu}(ay+xc).$$
This is invertible so long as $a^2-\beta c^2$ is a unit.  Since $\beta$ is a square, considering the reduction of the vector $\left( \begin{array}{c} a \\ c \end{array} \right)$ as an element of $\P^1(\Fp)$, we see that $p-1$ of the $p+1$ elements will cause $a^2-\beta c^2$ to be a unit.  Thus the total number of choices of $a,c \in \Zpn$ giving rise to invertible matrices is
$$(p-1)^2p^{2n-2};$$
incorporating $x,y$ as well, we conclude that the centralizer has order
$$(p-1)^2 p^{2n-2}p^{2\mu}.$$
As $\GL_2(\Zpn)$ has order $(p-1)^2 (p+1)p^{4n-3}$ the asserted formula for the size of the conjugacy class follows.

\end{proof}

It is not especially difficult to give an algorithm to determine which of these conjugacy classes any particular elements of $\GL_2(\Zpn)$ lies in.  Given
$$g=\mat{a}{b}{c}{d} \in \GL_2(\Zpn):$$
\begin{enumerate}

\item Compute
\begin{align*}
\sigma &= a+d; \\
\tau &= ad-bc; \\
\Delta &= \sigma^2-4\tau; \\
\mu &= \min\{v(a-d),v(b),v(c)\}.
\end{align*}
\item If $\mu \geq n$, then $g$ lies in $\text{I}(a)$.
\item If $\mu = 0$, then consider $\Delta$.
\begin{enumerate}
\item If $\Delta \not\in \Zpns$, then $g$ lies in 
$\text{II}(\frac{\sigma}{2},\frac{\Delta}{4p})$.
\item If $\Delta \in (\Zpn)^{\times 2}$, then
$g$ lies in $\text{III}(\alpha,\beta)$ with
\begin{align*}
\alpha &= \frac{\sigma-\sqrt{\Delta}}{2} \\
\beta &= \frac{\sigma+\sqrt{\Delta}}{2}.
\end{align*}
\item If $\Delta \in \Zpns$ is not a square, then
$g$ lies in $\text{IV}(-\tau,-\sigma)$.
\end{enumerate}
\item If $0 < \mu < n$, then set $\alpha = \frac{\sigma}{2}$ and $\nu=v(\Delta)-\mu$.  If $\nu =n$, then $g$ lies in $\text{I}'_\mu(\alpha)$.  Otherwise set
$$\beta = \frac{\Delta}{4p^{\mu+\nu}}.$$
\begin{enumerate}
\item If $\mu < \nu < n$, then $g$ lies in
$\text{I}'_{\mu,\nu}(\alpha,\beta)$.
\item If $\mu=\nu$, then $g$ lies in
$\text{I}^{\pm}_{\mu}(\alpha,\beta)$ where the sign
equals the Legendre symbol $\left(\frac{\beta}{p}\right)$.
\end{enumerate}

\end{enumerate}

\section{Double coset types: unramified case}

\subsection{Preliminaries}

Fix an odd integer $N$.  Let $G=\GL_2(\Z/N)$ and
$$\Gamma = \left\{ \mat{1}{*}{0}{*} \right\} \subseteq G.$$
Let $L/E$ be an extension of number fields with $\Gal(L/E) \cong G$ and set
$K = L^\Gamma$.  Our primary goal is to understand prime decompositions for the extension
$K/E$.  By Lemma~\ref{lemma:factor}, it is equivalent to understand the double coset types
$D \backslash G / \Gamma$ as $D$ varies over all possible decomposition groups in $G$.  In this section we give a complete answer to this question in the unramified case, so that $D$ may be any cyclic subgroup of $G$.

\begin{lemma} \label{l:coset}
The cosets $G/\Gamma$ are of the form
$$\C_{a,c} = \left\{ \mat{a}{*}{c}{*} \right\}.$$
\end{lemma}
\begin{proof}
If $g=\mat{a}{b}{c}{d} \in G$ and $\gamma=\mat{1}{x}{0}{y} \in \Gamma$, then
$$g\gamma = \mat{a}{ax+by}{c}{cx+dy}$$
so it is enough to show that any element of $G$ with first column $\left(\begin{array}{c}a \\ c \end{array}\right)$ can be written in this way.  This is obvious as $g$ acts bijectively on the set of column vectors.
\end{proof}

Recall that if $N=p_1^{n_1}\cdots p_m^{n_m}$, then
$$\GL_2(\Z/N) \cong \GL_2(\Z/p_1^{n_1}) \times \cdots \GL_2(\Z/p_m^{n_m})$$
and
$$\Gamma \cong \Gamma_1 \times \cdots \times \Gamma_m$$
where $\Gamma_i$ is the $\GL_2(\Z/p_i^{n_i})$-analogue of $\Gamma$.  By
Lemma~\ref{lemma:dcpprod} we may thus reduce to the case when $N$ is an odd prime power.

For the rest of this section we fix a power $p^n$ of an odd prime $p$.
Fix $g \in \GL_2(\Z/p^n)$.
By Lemma~\ref{l:coset}, to study the double coset type $\left< g \right> \backslash G / \Gamma$ it suffices to study the action of $g$ on the set 
$$W = \left\{ \vec{a}{c} ; a,c \in \Zpn, \text{~at least one of~}a,c \text{~in~}\Zpns \right\}.$$
Writing $\lambda_k(g)$ for the number of elements of $W$ lying in orbits of order precisely $k$ for the action of $\left< g \right>$, we see that the double coset type of $g$ is simply
\begin{equation} \label{eq:lkg}
\sum_{k=1}^{\infty} \frac{\lambda_k(g)}{k} \times k.
\end{equation}

We study the arithmetic function $\lambda_k(g)$ as follows.
Let $V$ denote a free $\Zpn$-module of rank $2$.  We let $G$ act on $V$ by left multiplication.  Given $g \in G$, define
$\ell_k^t(g)$ to be the $\Zpn$-module length of the kernel of the endormorphism
$$g^k-1 : p^t V \to p^t V.$$
Since $W = V - pV$, we see that
$$\lt_{k}(g) := p^{\ell_k^0(g)} - p^{\ell_k^1(g)}$$
equals the number of elements of $W$ lying in orbits of order dividing $k$.  Thus
$$\lt_k(g) = \sum_{d \mid k} \lambda_d(g),$$
so that by M\"obius inversion we conclude that
$$\lambda_k(g) = \sum_{d \mid k} \mu\left(\frac{k}{d}\right) \lt_d(g).$$
Thus to compute our desired double coset types it suffices (in principle) to compute $\ell_k^t(g)$ for $t=0,1$ and all $g$ and $k$.

\subsection{Unramified double coset types: statement}

Given $\alpha \in \Zp^\times$, we write $\oa$ for the order of $\alpha$ in $\Zpns$ and
$\oab$ for the order of $\alpha$ in $\Fp^\times$.    We write $v$ for the usual valuation on $\Zp$.

\begin{lemma} \label{lemma:val}
Fix $\alpha \in \Zp^\times$ and $k \geq 1$.  Set
$v_\alpha = v(\alpha^{\oab}-1)$; note that $v_\alpha > 0$.  Then
$$v(\alpha^k-1) = \begin{cases} 0 & \oab \nmid k \\
\voa+v(k) & \oab \mid k.\end{cases}$$
Furthermore,
$$\oa = \oab p^{\max\{n-\voa,0\}}.$$  In particular, if $\voa \leq n$, then
$\voa = n- v(\oa)$.
\end{lemma}
\begin{proof}  By definition we can write
$$\alpha^{\oab} = 1+p^{\voa} \beta$$
for some $\beta \in \Zp^\times$.  Then
\begin{align*}
\left(\alpha^{\oab}\right)^{p^s}  &= (1+p^{\voa} \beta)^{p^s} \\
&= 1 + p^{\voa+s}\beta + \cdots
\end{align*}
which proves the first statement.  The second statement follows immediately since $\alpha^k=1$ in $\Zpns$ if and only if $v(\alpha^k-1) \geq n$.
\end{proof}

\begin{prop} \label{prop:main}
For $g \in G$, the double coset type of $\left< g \right> \backslash G / \Gamma$ depends only on the conjugacy class of $g$ and is given in Tables~\ref{t:dcp2} and ~\ref{t:dcp}.
\end{prop}

\begin{rmk} $u_1,u_2,u_3$ are valuations in terms of certain $p$-adic numbers $\zkm(\alpha)$ defined in Lemma~\ref{l:super}.  The valuation $v(\zkm(\alpha))$ equals $2\voa-\mu$, so that
in most cases $v_i$ is clear without computing $\zkm(\alpha)$.
\begin{enumerate}
\item $u_1 = v(p^\nu\beta-\zkm(\alpha))$ so that
$$u_1 = \begin{cases} \nu & \nu < 2\voa-\mu; \\
2\voa-\mu & \nu > 2\voa-\mu. \end{cases}$$
\item $u_2 = v(p^\mu \beta - \zkm(\alpha))$ so that
$$u_2 = \begin{cases} \mu & \mu < \voa ; \\
2\voa-\mu & \mu >\voa. \end{cases}$$
\item $u_3 = v(p\beta - z^0(\alpha))$ so that
$$u_3 = \begin{cases} v(\beta)+1 & v(\beta) < 2\voa-1; \\
2\voa & v(\beta) > 2\voa-1. \end{cases}$$
\item $u_4 = \min\bigl\{2v(6\beta-2\alpha^2-2\varepsilon \alpha+1)-1,
2v(2\alpha+\varepsilon),2n-2\bigr\}$ with $\varepsilon=1$ (resp.\ $-1$) if $\alpha \equiv 1 \pmod{3}$
(resp.\ $\alpha \equiv 2 \pmod{3}$).
\end{enumerate}
\end{rmk}

\begin{rmk}
In case $\text{IV}$, $o_{\bar{g}}p^{v_g}$ equals the order of any element of the conjugacy class, with $p \nmid o_{\bar{g}}$.
\end{rmk}

\begin{table} 
\renewcommand{\arraystretch}{2}
$$\begin{array}{|l|l|l|} \hline
\text{Class} & \text{Condition} & \text{type} \\ \hline
\text{I}(\alpha) & & \dcp(\oab;\voa) \\ \hline

\text{I}'_{\mu,\nu}(\alpha,\beta) & \voa \leq \mu & \dcp(\oab;\voa,u_1+\mu-\voa) \\
& \voa \geq \mu & \dcp(\oab;\mu,u_1) \\ \hline

\text{I}'_\mu(\alpha) & \voa \leq \mu & \dcp(\oab;\voa) \\
& \voa \geq \mu & \dcp(\oab;\mu,2\voa-\mu) \\ \hline

\text{I}^\pm_\mu(\alpha,\beta) & \voa \neq \mu & \dcp(\oab;\min\{\voa,\mu\}) \\
& \voa = \mu & \dcp(\oab;\mu,u_2) \\ \hline

\text{II}(\alpha,\beta) & p \neq 3 \text{~or~} v(\beta) > 0& \dcp(\oab;0,u_3) \\ 
 & p = 3,  v(\beta)=0, 2 \mid u_4 &
\dcp(\oab;\frac{u_4}{2}) \\ 
 & p = 3,   v(\beta)=0, 2 \nmid u_4  &
\dcp(\oab;\frac{u_4-1}{2},\frac{u_4+1}{2}) \\
\hline

\text{III}(\alpha,\beta) & \oab = \obb, \voa=\vob & \dcp(\oab;\voa) \\
& \oab = \obb, \voa < \vob & \dcp(\oab;\voa,\vob) \\
& \oab \neq \obb, \voa \leq \vob & \dcp(\oab,\obb;\voa,\vob) \\ \hline

\text{IV}(\alpha,\beta) & & \dcp(o_{\bar{g}};v_g) \\ \hline
\end{array}$$
\renewcommand{\arraystretch}{1}
\caption{Double coset types by conjugacy class} \label{t:dcp2}
\end{table}

\begin{table} 
\renewcommand{\arraystretch}{2}
$$\begin{array}{|l|l|l|}  \hline
& \text{Condition} & \text{Double Coset type} \\ \hline
\dcp(k_0;a) & & \frac{(p^2-1)p^{n+a-2}}{k_0} \times k_0 p^{n-a} \\ \hline
\dcp(k_0;a,b) & a < b & \frac{(p-1)p^{n+a-1}}{k_0} \times k_0 p^{n-b} + {\displaystyle \sum_{u=n-b+1}^{n-a-1}}\frac{(p-1)^2p^{n+a-2}}{k_0}  \times k_0 p^{u} +\\
& &  \frac{(p-1)p^{n+a-1}}{k_0} \times k_0 p^{n-a} \\ \hline
\dcp(k_1,k_2;a) & k_1 \mid k_2 & \frac{(p-1)p^{a-1}}{k_1} \times k_1p^{n-a} + \frac{(p-1)(p^n+p^{n-1}-1)p^{a-1}}{k_2} \times k_2p^{n-a} \\ \hline
\dcp(k_1,k_2;a,b) & k_1 \mid k_2 & \frac{(p-1)p^{a-1}}{k_1} \times k_1 p^{n-a} + \frac{(p-1)p^{n+a-1}}{k_2} \times k_2p^{n-b} + \\
& a < b & {\displaystyle \sum_{u=n-b+1}^{n-a-1}} \frac{(p-1)^2p^{n+a-2}}{k_2} \times k_2 p^u + 
\frac{(p-1)(p^n-1)p^{a-1}}{k_2} \times k_2 p^{n-a} \\ \hline
\dcp(k_1,k_2;a) & k_2 \mid k_1 & \frac{(p-1)p^{a-1}}{k_2} \times k_2p^{n-a} + \frac{(p-1)(p^n+p^{n-1}-1)p^{a-1}}{k_1} \times k_1p^{n-a} \\ \hline
\dcp(k_1,k_2;a,b) & k_2 \mid k_1 &\frac{(p-1)p^{a-1}}{k_2} \times k_2p^{n-a} + \frac{(p-1)(p^{n-b+a}-1)p^{b-1}}{k_1} \times k_1p^{n-b} + \\
& a < b & {\displaystyle \sum_{u=n-b+1}^{n-a-1}} \frac{(p-1)^2p^{n+a-2}}{k_1} \times k_1 p^u +
\frac{(p-1)(p^n+p^{n-1}-1)p^{a-1}}{k_1} \times k_1 p^{n-a} \\ \hline
\dcp(k_1,k_2;a) & k_1 \nmid k_2 & \frac{(p-1)p^{a-1}}{k_1} \times k_1p^{n-a} + \frac{(p-1)p^{a-1}}{k_2} \times k_2p^{n-a} + \\
& k_2 \nmid k_1 & \frac{(p-1)(p^n+p^{n-1}-2)p^{a-1}}{k_3} \times k_3p^{n-a} \\ \hline
\dcp(k_1,k_2;a,b) & k_1 \nmid k_2 & \frac{(p-1)p^{a-1}}{k_1} \times k_1p^{n-a} + \frac{(p-1)p^{b-1}}{k_2} \times k_2p^{n-b} + \\
& k_2 \nmid k_1 & \frac{(p-1)(p^{n-b+a}-1)p^{b-1}}{k_3} \times k_3p^{n-b} + 
{\displaystyle \sum_{u=n-b+1}^{n-a-1}} \frac{(p-1)^2p^{n+a-2}}{k_3} \times k_3p^u + \\
& a < b &\frac{(p-1)(p^n-1)p^{a-1}}{k_3} \times k_3 p^{n-a} \\ \hline
\end{array}$$
\renewcommand{\arraystretch}{1}
\caption{Standard double coset types} \label{t:dcp}
\end{table}

The proof is primarily a long exercise in arithmetic functions, Smith normal forms and Newton polygons.  We give the details in the succeeding sections.   Recall that the {\it Smith normal form} of a matrix
 over $\Zp$ is a diagonal matrix with powers of $p$ on the diagonal which can be obtained via (independent) row and column operations.  In the case of $2 \times 2$ matrices, there is no need to consider the actual operations: the 
Smith normal form of $\mat{a}{b}{c}{d}$ is simply
\begin{equation} \label{eq:smith}
\mat{p^{\min\{v(a),v(b),v(c),v(d)\}}} {0}{0}{p^{v(ad-bc)-\min\{v(a),v(b),v(c),v(d)\}}}.
\end{equation}
Furthermore, the exponent in the upper left will always be less than or equal to that in the lower right.
The same discussion holds over $\Z/p^n$ with the obvious modifications.

Note that if $g^k-1$ has Smith normal form
$$\mat{p^{a}}{0}{0}{p^{b}},$$
then we have
\begin{align*}
\ell_k^0(g) &= p^{\max\{a,n\}+\max\{b,n\}} \\
\ell_k^1(g) &=  p^{\max\{a,n-1\}+\max\{b,n-1\}}.
\end{align*}
In particular, $\lt_k(g)=0$ unless at least one of $a,b$ equals $n$.  We will study arithmetic functions of this form in Section~\ref{ss:af}.

\subsection{Examples}

Take $p=5$ and $n=4$.  Consider first the matrix
$$\mat{32}{20}{5}{32}.$$
This lies in the conjugacy class $\text{I}^+_1(32,4)$ of order $18750$.
We have $\oab = 4$ and $\voa = 4-v(\oa)=4-v(100)=2$.  Therefore the double coset type is
$$\dcp(4;\min\{2,1\}) = \dcp(4;1) = 750 \times 500.$$

Consider next
$$\mat{2}{20}{5}{2}$$
in the conjugacy class $\text{I}^+_1(2,4).$
This time $\oab=4$ and $\voa=4-v(\oa)=4-v(500)=1$.  Since $\voa = \mu$, we are in the case where we must compute
the $5$-adic integer $z^1(2)$ in order to compute the double coset type.  It is the root of largest valuation of 
$$t_4^1(2,x) = 625x^4-2000x^3+2350x^2-1520x+225$$
and one finds that
$$z^1(2) = 5 + 4 \cdot 5^2 + 1 \cdot 5^3 + \cdots.$$
(In fact, it is easily verified that $\frac{1}{5}$ and $\frac{9}{5}$ are roots of 
$t_4^1(2,x)$, so that $z^1(2)$ is a root of the quotient $25x^2-30x+25$; thus
$$z^1(2) = \frac{3}{5} \pm \frac{4}{5}\sqrt{-1}.$$
Keep in mind that, depending on the choice of $\sqrt{-1}$, one choice of sign gives the desired $z^1(2)$ of valuation 1 while the other gives an element of valuation $-1$.)
Thus 
\begin{align*}
u_2 &= v\bigl(p^\mu \beta - z^1(2)\bigr) \\
&= v\bigl(20-z^1(2)\bigr) \\
&= v\bigl(3 \cdot 5 + 1 \cdot 5^2 + \cdots\bigr) \\
&= 1
\end{align*}
and the double coset type is still
$$\dcp(4;1,1) = \dcp(4;1) = 750 \times 500.$$

However, if we have 
$$\mat{2}{5}{5}{2}$$
so that $\beta=1$, we instead find that
\begin{align*}
u_2 &= v\bigl(p^\mu \beta - z^1(2)\bigr) \\
&= v\bigl(5-z^1(2)\bigr) \\
&= v\bigl(5^2 + 3 \cdot 5^3 + \cdots\bigr) \\
&= 2
\end{align*}
resulting in a double coset type of
$$\dcp(4;1,2) = 625 \times 100 + 625 \times 500.$$

As $p\beta$ becomes closer to $z^1(2)$, the double coset type gains double cosets of smaller order.  The matrix
$$\mat{2}{105}{5}{2}$$
with $\beta = 21$ has $u_2=3$ and thus type
$$\dcp(4;1,3) = 625 \times 20 + 500 \times 100 + 625 \times 500.$$
We can go one more step: taking $\beta=46$, we have $u_2 = 4$ so that the matrix
$$\mat{2}{230}{5}{2}$$
has double coset type
$$\dcp(4;1,4) = 625 \times 4 + 500 \times 20 + 500 \times 100 + 625 \times 500.$$

\subsection{Arithmetic functions} \label{ss:af}

We will encounter only  a limited collection of double coset types.  In this section we work out those standard types.  We need work here only with the arithmetic functions, so assume we are given
arithmetic functions
$$\ell_{\cdot}^0, \hspace{.5em} \ell_{\cdot}^1 : \N \to \N.$$
We then define two more arithmetic functions
$$\lt_{\cdot}, \hspace{0.5em} \lambda_{\cdot} : \N \to \N$$
by
$$\lt_k = p^{\ell_k^0} - p^{\ell_k^1}$$
and 
$$\lambda_k = \sum_{d \mid k} \mu(d)\lt_{\frac{k}{d}}.$$

We are interested in converting the expressions we will encounter for $\ell^0$ and $\ell^1$ to expressions for $\lambda$.
We will most often find ourselves in the situation of the next lemma; in terms of Smith normal forms, this would correspond to the case that
$g^k-1$ has Smith normal form
$$\mat{p^{v(k)+a}}{0}{0}{p^{v(k)+b}}$$
for some $a,b$ and all $k$ divisible by a fixed $k_0$.

 For an integer $k_0$, let
$e_{k_0}(k)$ denote the function which is $1$ if $k_0 \mid k$ and $0$ if $k_0 \nmid k$.

\begin{lemma} \label{lemma:mu}
Suppose that there is $k_0$ relatively prime to $p$ and integers
$$0 \leq a \leq b \leq n$$
such that
\begin{align*}
\ell_k^0 &= e_{k_0}(k) \cdot \bigl(
\min\{v(k)+a,n\} + \min\{v(k)+b,n\} \bigr); \\
\ell_k^1 &= e_{k_0}(k) \cdot \bigl(
\min\{v(k)+a,n-1\} + \min\{v(k)+b,n-1\} \bigr).
\end{align*}
If $a = b$, then
$$\lambda_k = \begin{cases} 0 & k \neq k_0 p^{n-a}; \\
p^{2n}-p^{2n-2} & k = k_0 p^{n-a}.
\end{cases}$$
If $a < b$, then
$$\lambda_k = \begin{cases} 0 & k \neq k_0 p^u  \text{~with~} n-b \leq u \leq n-a \\
(p-1)p^{2n+a-b-1} & k = k_0 p^{n-b}; \\
(p-1)^2 p^{n+u+a-2} & k = k_0p^u \text{~with~} n-b < u < n-a; \\
(p-1)p^{2n-1} & k = k_0 p^{n-a}.
\end{cases}$$
\end{lemma}

\begin{proof}
Assume first that $a=b$.  Then $\ell_k^0=\ell_k^1$ for all $k < k_0 p^{n-a}$, while $\ell_{k_0 p^{n-a}}^0= 2n$ and $\ell_{k_0 p^{n-a}}^{1}=2n-2$.  Thus
$$
\lt_k= \begin{cases} 0 & k < k_0 p^{n-a}; \\ p^{2n}-p^{2n-2} & k=k_0 p^{n-a}.
\end{cases}$$
The asserted formula for $\lambda_k$ in this case follows immediately.

Assume then that $a<b$.  
Note that $\ell_k^0 = \ell_k^1$ (and thus $\lt_k=0$) unless $v(k) +b \geq n$.
Therefore
$$\lt_k = \begin{cases} 0 & k_0 \nmid k \text{~or~} v(k) < n-b \\
p^{n+v(k)+a}-p^{n+v(k)+a-1} & k_0 \mid k, n-b \leq v(k) < n-a \\
p^{2n}-p^{2n-2} & k_0 \mid k, n-a \leq v(k). \end{cases}$$
It follows immediately from M\"obius inversion that $\lambda_k=0$ unless $k_0 \mid k$.

Fix now $k$ divisible by $k_0$ and write $k=k_0 k' p^u$ with $p \nmid k'$.  We compute
\begin{align*}
\lambda_k &= \sum_{d \mid k_0k' p^u} \mu(d)\lt_{k/d}\\
&= \sum_{d \mid k' p^u} \mu(d)\lt_{k/d}\\ \intertext{(since $\lt_k = 0$ if $k_0 \nmid k$)}
&= \sum_{d \mid k'} \mu(d)\lt_{k/d} - \sum_{d \mid k'} \mu(d)\lt_{k/pd} \\
&= \left( \sum_{d \mid k'} \mu(d) \right) \lt_{k} - \left( \sum_{d \mid k'} \mu(d) \right) \lt_{k/p} \\ \intertext{(since $\lt$ is constant in each sum)}
&= \begin{cases} 0 & k' \neq 1 \\
\lt_{k} - \lt_{k/p} & k' = 1
\end{cases}
\end{align*}
by a standard property of the M\"obius function.  Therefore $\lambda_k=0 $ for $k \neq k_0 p^{u}$ with $n-b \leq u \leq n-a$, while the formulae for the remaining $u$ follow easily from the formula for
$\lt_k$.  For example, if $n-b < u < n-a$, then
\begin{align*}
\lambda_{k_0p^u} &= \lt_{k_0 p^u}- \lt_{k_0 p^{u-1}}\\
&= p^{n+u+a}-p^{n+u+a-1}-p^{n+u+a-1}+p^{n+u+a-2} \\
&= (p-1)^2p^{n+u+a-2}.
\end{align*}
\end{proof}

\begin{cor} \label{cor:dcp}
Suppose there is $k_0$ relatively prime to $p$ and $a \leq b$ such that $g^k-1$ has Smith normal form the identity if $k_0 \nmid k$ and
$$\mat{p^{v(k)+a}}{0}{0}{p^{v(k)+b}}$$
if $k_0 \mid k$.
If $a=b$, then the double coset type of $g$ is that given in Table~\ref{t:dcp} as $\dcp(k_0;a)$.
If $a < b$, then the double coset type of $g$ is that given in Table~\ref{t:dcp} as $\dcp(k_0;a,b)$.
\end{cor}
\begin{proof}
Restricting to $k$ divisible by $k_0$, we have
\begin{align*}
\ell_k^0(g) &= \min\{v(k)+a,n\} + \min\{v(k)+b,n\};\\
\ell_k^1(g) &= \min\{v(k)+a,n-1\} + \min\{v(k)+b,n-1\}.
\end{align*}
The corollary therefore follows immediately from Lemma~\ref{lemma:mu} and (\ref{eq:lkg}).
\end{proof}

Unfortunately we will also have need of the following more complicated version. 

\begin{lemma} \label{lemma:diag}
Suppose that there are $k_1$, $k_2$, distinct and relatively prime to $p$, and integers
$0 \leq a \leq b \leq n$
such that
\begin{itemize}
\item $\ell_k^0 = e_{k_1}(k)\min\{v(k)+a,n\} + e_{k_2}(k)\min\{v(k)+b,n\}$
\item $\ell_k^1 = e_{k_1}(k)\min\{v(k)+a,n-1\} + e_{k_2}(k)\min\{v(k)+b,n-1\}$
\end{itemize}
Let $k_3$ denote the least common multiple of $k_1$ and $k_2$.
Then the double coset types corresponding to the associated function
$\lambda$ via (\ref{eq:lkg}) are as given in Table~\ref{t:dcp}.
\end{lemma}

\begin{proof}
We have
\begin{multline*}
\lt_k = p^{e_{k_1}(k)\min\{v(k)+a,n\} + e_{k_2}(k)\min\{v(k)+b,n\}} - \\
p^{e_{k_1}(k)\min\{v(k)+a,n-1\} + e_{k_2}(k)\min\{v(k)+b,n-1\}}.
\end{multline*}
Writing $k=k' p^u$ with $p \nmid k'$, we have
\begin{align*}
\lambda_k &= \sum_{d \mid k} \mu(d)\lt_{k/d} \\
&= \sum_{d \mid \frac{k'}{k_1}p, d \nmid \frac{k'}{k_2}p} \mu(d)\lt_{k/d} +
\sum_{d \mid \frac{k'}{k_2}p, d \nmid \frac{k'}{k_1}p} \mu(d)\lt_{k/d} +
\sum_{d \mid \frac{k'}{k_3}p} \mu(d)\lt_{k/d} \\
\intertext{breaking up the sum and using that $\lt_{k/d}=0$ unless some $k_i$ divides $\frac{k'}{d}$; we use the convention that $d \mid x$ is always false if $x \notin \Z$, so that some of these sums may be empty depending on $k_1$ and $k_2$;}
&= \sum_{d \mid \frac{k'}{k_1}, d \nmid \frac{k'}{k_2}} \mu(d)\left(\lt_{k_1p^u}-\lt_{k_1p^{u-1}}\right) +
\sum_{d \mid \frac{k'}{k_2}, d \nmid \frac{k'}{k_1}} \mu(d)\left(\lt_{k_2p^u}-\lt_{k_2p^{u-1}}\right) + \\
& \sum_{d \mid \frac{k'}{k_3}} \mu(d)\left(\lt_{k_3p^u}-\lt_{k_3p^{u-1}}\right).
\end{align*}
Note that
\begin{align*}
\sum_{d \mid \frac{k'}{k_1}, d \nmid \frac{k'}{k_2}} \mu(d) &= \sum_{d \mid \frac{k'}{k_1}} \mu(d) -
\sum_{d \mid \frac{k'}{k_1}, d \mid \frac{k'}{k_2}} \mu(d) \\
&= \sum_{d \mid \frac{k}{k_1}} \mu(d) -
\sum_{d \mid \frac{k'}{k_3}} \mu(d) \\
&= \begin{cases} 1 & k_1=k' \neq k_3; \\
-1 & k_1 \neq k' = k_3;\\
0 & \text{otherwise}. \end{cases} \\
&= \begin{cases} 1 & k'=k_1 \text{~and~} k_2 \nmid k_1; \\
-1 & k' = k_3 \text{~and~} k_2 \nmid k_1;\\
0 & \text{otherwise} \end{cases}
\end{align*}
and similarly for the second sum.

If $k_2 \mid k_1$, then $k_3 = k_1$.  We therefore conclude that
\begin{align*}
\lambda_{k_2 p^u} &= \lt_{k_2 p^u} - \lt_{k_2 p^{u-1}} \\
&= p^{\min\{u+b,n\}} -p^{\min\{u+b,n-1\}} - p^{\min\{u+b-1,n\}} + p^{\min\{u+b-1,n-1\}} \\
&= \begin{cases} (p-1)p^{n-1} & u=n-b \\ 0 & u \neq n-b \end{cases} \\
\lambda_{k_1 p^u} &= \lt_{k_1 p^u} - \lt_{k_1 p^{u-1}} - \lt_{k_2 p^{u}} + \lt_{k_2 p^{u-1}} \\
&= p^{\min\{u+a,n\} + \min\{u+b,n\}} - 
p^{\min\{u+a,n-1\} + \min\{u+b,n-1\}} - \\
& p^{\min\{u+a-1,n\} + \min\{u+b-1,n\}} +
p^{\min\{u+a-1,n-1\} + \min\{u+b-1,n-1\}} - \\
& p^{\min\{u+b,n\}} + p^{\min\{u+b,n-1\}} + p^{\min\{u+b-1,n\}} - p^{\min\{u+b-1,n-1\}} \\
&= \begin{cases} (p-1)(p^{n-b+a}-1)p^{n-1} & u=n-b, a < b \\
(p-1)^2 p^{n+u+a-2} & n-b < u < n-a \\
(p-1)p^{2n-1} & u=n-a, a < b \\
(p-1)(p^{n}+p^{n-1}-1)p^{n-1} & u=n-a, a = b  \end{cases}
\end{align*}
and all other values are zero.  This results in the asserted double coset types in these two cases.

If $k_1 \mid k_2$, then $k_3 = k_2$.  Therefore
\begin{align*}
\lambda_{k_1 p^u} &= \lt_{k_1 p^u} - \lt_{k_1 p^{u-1}} \\
&= p^{\min\{u+a,n\}}-p^{\min\{u+a,n-1\}}-p^{\min\{u+a-1,n\}}+p^{\min\{u+a-1,n-1\}} \\
&= \begin{cases} (p-1)p^{n-1} & u=n-a; \\
0 & \text{otherwise} \end{cases} \\
\lambda_{k_2 p^u} &= \lt_{k_2 p^u}-\lt_{k_2 p^{u-1}}-\lt_{k_1 p^u}+\lt_{k_1 p^{u-1}} \\
&= p^{\min\{u+a,n\}+\min\{u+b,n\}} - p^{\min\{u+a,n-1\}+\min\{u+b,n-1\}} - \\
& p^{\min\{u+a-1,n\}+\min\{u+b-1,n\}} + p^{\min\{u+a-1,n-1\}+\min\{u+b-1,n-1\}} - \\
& p^{\min\{u+a,n\}} + p^{\min\{u+a,n-1\}} + p^{\min\{u+a-1,n\}} - p^{\min\{u+a-1,n-1\}} \\
&= \begin{cases}
(p-1)p^{2n+a-b-1} & u=n-b, a < b; \\
(p-1)^2p^{u+a+n-2} & n-b < u < n-a; \\
(p-1)(p^n-1)p^{n-1} & u=n-a, a < b; \\
(p-1) (p^{n}+p^{n-1}-1   ) p^{n-1} & u=n-a, a = b; \end{cases}
\end{align*}
and all other values are zero.

Finally, if neither $k_1$ nor $k_2$ divides the other, we have
\begin{align*}
\lambda_{k_1p^u} &= \lt_{k_1 p^u}-\lt_{k_1 p^{u-1}} \\
&= p^{\min\{u+a,n\}}-p^{\min\{u+a,n-1\}}-p^{\min\{u+a-1,n\}}+p^{\min\{u+a-1,n-1\}} \\
&= \begin{cases} (p-1)p^{n-1} & u=n-a; \\
0 & \text{otherwise} \end{cases} \\
\lambda_{k_2 p^u} &= \lt_{k_2 p^u} - \lt_{k_2 p^{u-1}} \\
&= p^{\min\{u+b,n\}} -p^{\min\{u+b,n-1\}} - p^{\min\{u+b-1,n\}} + p^{\min\{u+b-1,n-1\}} \\
&= \begin{cases} (p-1)p^{n-1} & u=n-b \\ 0 & u \neq n-b \end{cases} \\
\lambda_{k_3p^u} &= \lt_{k_3p^u}-\lt_{k_3p^{u-1}}-\lt_{k_1 p^u}+\lt_{k_1p^{u-1}}-\lt_{k_2 p^u}+\lt_{k_2 p^{u-1}} \\
&= p^{\min\{u+a,n\}+\min\{u+b,n\}} - p^{\min\{u+a,n-1\}+\min\{u+b,n-1\}} - \\
& p^{\min\{u+a-1,n\}+\min\{u+b-1,n\}} + p^{\min\{u+a-1,n-1\}+\min\{u+b-1,n-1\}} - \\
& p^{\min\{u+a,n\}} + p^{\min\{u+a,n-1\}} + p^{\min\{u+a-1,n\}} - p^{\min\{u+a-1,n-1\}} -\\
& p^{\min\{u+b,n\}} + p^{\min\{u+b,n-1\}} + p^{\min\{u+b-1,n\}} - p^{\min\{u+b-1,n-1\}} \\
&= \begin{cases}
(p-1)(p^{n-b+a}-1)p^{n-1} & u=n-b, a < b; \\
(p-1)^2p^{n+u+a-2} & n-b < u < n-a; \\
(p-1)(p^n-1)p^{n-1} & u=n-a, a<b; \\
(p-1)(p^{n}+p^{n-1}-2p)p^{n-1} & u=n-a, a=b; \end{cases}
\end{align*}
and is zero otherwise.

\end{proof}

\subsection{Newton Polygons} \label{ss:np}

As is clear from (\ref{eq:smith}),
we will have repeated need to compute valuations of certain determinants.  We consider the general situation in this section.

Consider a matrix
$$g = \mat{\alpha}{x}{p^\mu}{\alpha}$$
with $\alpha \in \Zp^\times$, $\mu \geq 0$ and $x$ an indeterminate.
One checks easily that
$$g^k = \mat{\rkm(\alpha,x)}{x \skm(\alpha,x)}{p^\mu \skm(\alpha,x)}{\rkm(\alpha,x)}$$
with
\begin{align*}
\rkm(\alpha,x) &= \sum_{i=0}^{\frac{k}{2}} \binom{k}{2i}\alpha^{k-2i} p^{\mu i} x^i\\
\skm(\alpha,x) &= \sum_{i=0}^{\frac{k-1}{2}} \binom{k}{2i+1}\alpha^{k-2i-1} p^{\mu i}x^i.
\end{align*}
Define
$$\tkm(\alpha,x) = \det (g^k-1) = (\rkm(\alpha,x)-1)^2 - x p^\mu \skm(\alpha,x)^2 \in \Zp[x].$$

\begin{lemma} \label{l:super}
If $\oab$ does not divide $k$, then all roots of  $\tkm(\alpha,x)$ have valuation $-\mu$.  If $\oab$ divides $k$,
then
$\tkm(\alpha,x)$ has a unique root  $z=\zkm(\alpha)$ of maximal valuation. It lies in $\Qp$, is independent of $k$ and has $v(z) = 2\voa-\mu$.  All other roots of $\tkm(\alpha,x)$ have valuation at most $\frac{2}{p-1}-\mu$.
\end{lemma}
\begin{proof}
We begin by computing the Newton polygon of
$\tkm(\alpha,x)$:  if $\oab \nmid k$, then the vertices of the Newton polygon are
$$(0,0), \quad (k,\mu k);$$
while if $\oab \mid k$, then they are
\begin{multline*}
(0,2\voa+2v(k)), \quad (1,2v(k)+\mu), \quad (p, 2v(k)+\mu p-2), \ldots \\
(p^i,2v(k)+ \mu p^i-2i),  \ldots,  \quad (p^{v(k)},\mu p^{v(k)}), \quad (k,\mu k).
\end{multline*}

\begin{figure}
\includegraphics[width=8cm]{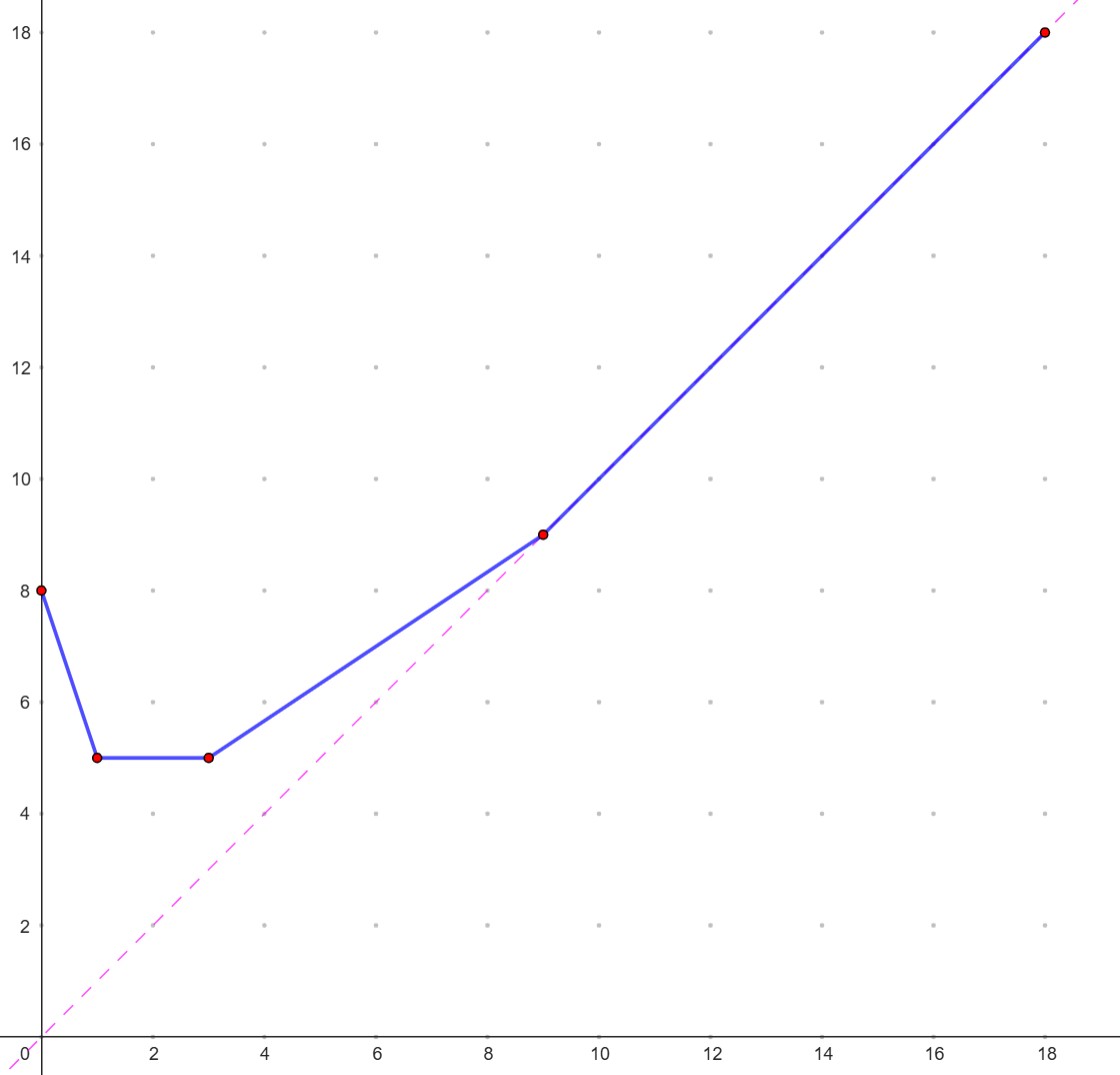}
\caption{The Newton polygon of $t_{18}^{1}(\alpha,x)$ with $\voa=2$ and $p=3$}
\end{figure}

One sees immediately that $t_k$ has degree $k$; write
$t_k = \sum_{m=0}^{k} c_m x^m$.   
Then
$$c_0 = (\alpha^k-1)^2$$
while
$$c_m = \left( 2(\alpha^k-1)\binom{k}{2m}\alpha^{k-2m} + \sum_{i = 1, i \neq 2m}^{k} (-1)^i
\binom{k}{i}\binom{k}{2m-i} \alpha^{2k-2m} \right) p^{\mu m}$$
for $m > 0$ (with the usual convention that $\binom{k}{i} = 0$ if $i < 0$ or $i > k$).

In particular, $v(c_m) \geq \mu m$ for all $m$ and
$$c_k = \pm p^{\mu k}$$
so that $v(c_k)=\mu k$.  If $\oa \nmid k$, then also $v(c_0)=0$.  Since the Newton polygon lies on or above the line
of slope $\mu$, it follows that the vertices must simply be $(0,0)$
and $(k,\mu k)$ in this case.

Assume now that $\oab \mid k$.  We have
$$c_1 = \left(2(\alpha^k-1)\binom{k}{2}\alpha^{k-2m} - k^2\alpha^{2k-2m} \right) p^\mu.$$
Since $v(\alpha^k-1)=\voa + v(k)$, we see that 
\begin{align*}
v(c_0) &= 2\voa+2v(k) \\
v(c_1) &= \mu+\min\{\voa+v(k)+v(k)+v(k-1),2v(k)\} = \mu+2v(k).
\end{align*}
If $p \nmid k$, then $v(c_1)=\mu$ lies on the line of slope $\mu$, so that as above we see that the Newton polygon
has vertices $(0,2\voa+2v(k))$, $(1,\mu)$, and $(k,\mu k)$.

Assume finally that $\oab p \mid k$.  
Since $p \mid k$, we have
$v\left(\binom{k}{i}\right) = v(k)-v(i)$ for $0 \leq i \leq k$.  Thus
\begin{align*}
v(c_m) &\geq \mu m+\min\Bigl\{ 2v(k)-v(1)-v(2m-1), 2v(k)-v(2)-v(2m-2), \ldots,\\
 & \qquad 2v(k)-2v(m) \Bigr\} \\
&= \mu m+2v(k) - \max \{ v(1)+v(2m-1), \ldots, 2v(m) \}.
\end{align*}
(The term $2(\alpha^k-1)\binom{k}{2m}\alpha^{k-2m}$ has valuation $\voa+2v(k)-v(2m)$ and can never be the term of least valuation since $\voa > 0$.)
If $m=p^i$ with $1 \leq i \leq v(k)$, then $2v(p^i)$ is larger than any other term above, so that we get an equality
$$v(c_{p^i}) = \mu p^i + 2v(k)-2i.$$
If $m$ satisfies $p^i < m < p^{i+1}$ with $i+1 \leq v(k)$, then the (not necessarily unique) largest value of
$v(j)+v(2m-j)$ occurs when $j=p^i$, where
$$v(p^i)+v(2m-p^i) = i + \min \{v(m),i\} \leq 2i.$$
Thus
$$v(c_m) \geq \mu m+2v(k)-2i.$$
One checks easily that this lies above the line segment connecting $(p^i,\mu p^i+2v(k)-2i)$ and $(p^{i+1},\mu p^{i+1}+2v(k)-2i-2)$, which confirms the vertices of the Newton polygon up to $x=p^{v(k)}$.  The final segment from $(p^{v(k)},\mu p^{v(k)})$ to
$(k,\mu k)$ is clear since the Newton polygon lies on or above the line of slope $\mu$.

Given the Newton polygons above, every statement in the lemma is clear except for the fact that the maximal valuation
roots of 
$\tkm(\alpha,x)$ for $\oab \mid k$ are all the same.  Let $\zkm(\alpha)$ denote the maximal valuation root of 
$\tom$.  Fix $k$ with $\oab \mid k$.  Since $\tkm$ is the determinant of $g^k-1$, which is divisible by $g^{\oab}-1$ in the
matrix algebra,
 the multiplicativity of the determinant shows that $\tom \mid \tkm$ in $\Zp[x]$.
Thus $\zkm(\alpha)$ is also a root of $\tkm$, as claimed.
\end{proof}

\begin{cor} \label{cor:val}
Fix $\alpha \in \Zp^\times$, $\mu \geq 0$ and $k \geq 1$. 
For any $\beta \in \Qp$ with $v(\beta) > \frac{2}{p-1}-\mu$ one has
$$v\bigl(\tkm(\alpha,\beta)\bigr) = \begin{cases} 0 & \oab \nmid k \\ v\bigl(\beta-\zkm(\alpha)\bigr)+2v(k)+\mu &
\oab \mid k. \end{cases}$$
\end{cor}
\begin{proof}
If $\oab \nmid k$, then this is immediate from the Newton polygon.
Assume therefore that $\oab \mid k$.  By Lemma~\ref{l:super} we may write
$$\tkm(\alpha,x) = p^{\mu k}\bigl(x-\zkm(\alpha)\bigr)(x-\gamma_1)\cdots(x-\gamma_{k-1})$$ with
$\zkm(\alpha) \in \Qp$ of valuation $2\voa-\mu$ and $\gamma_i \in \bar{\Q}_p$
of valuation at most $\frac{2}{p-1}-\mu$.
We compute
\begin{align*}
v\bigl(\tkm(\alpha,\beta)\bigr) &= \mu k + v\bigl(\beta-\zkm(\alpha)\bigr) + v(\beta - \gamma_1) + \cdots +
v(\beta - \gamma_{k-1}) \\
&=\mu k + v\bigl(\beta-\zkm(\alpha) \bigr) + v(\gamma_1) + \cdots +
v(\gamma_{k-1}) \\ \intertext{since $v(\beta) > v(\gamma_i)$ and thus $v(\beta-\gamma_i)=v(\gamma_i)$}
&= \mu k+ v\bigl(\beta-\zkm(\alpha)\bigr) + v(\gamma_1 \cdots \gamma_{k-1}) \\
&= \mu k + v\bigl(\beta-\zkm(\alpha)\bigr) + v\left( \frac{c_0}{p^{\mu k} \zkm(\alpha)} \right) \\
&= v\bigl(\beta-\zkm(\alpha)\bigr) + v(c_0) - v\bigl(\zkm(\alpha)\bigr) \\
&= v\bigl(\beta-\zkm(\alpha)\bigr) + 2\voa+2v(k)-2\voa+\mu \\
&= v\bigl(\beta-\zkm(\alpha)\bigr)+2v(k)+\mu.
\end{align*}

\end{proof}

Corollary~\ref{cor:val} will suffice in every case we consider except for $p=3$, $\mu=0$ and
$v(\beta)=1$.  For this we need the following refinement.

\begin{lemma} \label{l:val3}
Fix $\alpha \in \Z_3^\times$ and $\beta \in \Z_3$ with $v(\beta)=1$.
If $\alpha \equiv 1 \pmod{3}$, then
$$v\bigl(t_k^0(\alpha,\beta)\bigr) = \begin{cases} 1 & 3 \nmid k; \\
\min\{2v(2\beta-2\alpha^2-2\alpha+1)-1,2v(2\alpha+1)\} +2v(k) & 3 \mid k.\end{cases}$$
If $\alpha \equiv 2 \pmod{3}$, then
$$v\bigl(t_k^0(\alpha,\beta)\bigr) = \begin{cases} 0 & 2 \nmid k; \\
1 & (k,6)=2 \\
\min\{2v(2\beta-2\alpha^2+2\alpha+1)-1,2v(2\alpha-1)\} +2v(k) & 6 \mid k.\end{cases}$$
\end{lemma}
\begin{proof}
Assume first that $\alpha \equiv 1 \pmod{3}$.  If $3 \nmid k$, then the Newton polygon of the monic polynomial $t_k^0(\alpha,x)$ (as computed in the proof of Lemma~\ref{l:super}) shows that it has one root of valuation at least two and $k-1$ unit roots.  It follows immediately that
$v(t_k^0(\alpha,\beta))=1$.

Turning to the case that $3 \mid k$, one computes that
$$t_3^0(\alpha,x) = \bigl(x-(\alpha-1)^2\bigr)
\bigl(x^2+(1-2\alpha-2\alpha^2)x+(\alpha^2+\alpha+1)^2\bigr).$$
Since $\alpha \equiv 1 \pmod{3}$, we have
$$v\bigl(\beta-(\alpha-1)^2\bigr)=1.$$
The quadratic factor has roots
$$\gamma = \frac{2\alpha^2+2\alpha-1+(2\alpha+1)\sqrt{-3}}{2}$$
and its conjugate $\bar{\gamma}$.  Therefore
$$v(\beta-\gamma)=v(\beta-\bar{\gamma}) = \min\left\{ v\left(\beta-\frac{2\alpha^2+2\alpha-1}{2}\right),v(2\alpha+1)+\frac{1}{2} \right\}.$$
It follows that
$$v\bigl(t_3^0(\alpha,\beta)\bigr) = 1+\min\{2v(2\beta-2\alpha^2-2\alpha+1),2v(2\alpha+1)+1\}.$$

For general $k$ with $3 \mid k$, the Newton polygon shows that
$$t_k^0(\alpha,x) = t_3^0(\alpha,x)(x-\gamma_1)\cdots(x-\gamma_\ell)$$
where $v(\gamma_i)<1$.  Therefore
\begin{align*}
v\bigl(t_k^0(\alpha,\beta)\bigr) &= v\bigl(t_3^0(\alpha,\beta)\bigr)+v(\beta-\gamma_1)+\cdots
+v(\beta - \gamma_\ell) \\
&= v\bigl(t_3^0(\alpha,\beta)\bigr)+ v(\gamma_1) + \cdots + v(\gamma_\ell) \\
&= v\bigl(t_3^0(\alpha,\beta)\bigr)+v(\gamma_1 \cdots \gamma_\ell) \\
&= v\bigl(t_3^0(\alpha,\beta)\bigr)+v\left( \frac{(\alpha^k-1)^2}{(\alpha^3-1)^2}\right) \\
&= v\bigl(t_3^0(\alpha,\beta)\bigr) +2v(k)-2.
\end{align*}

The case of $\alpha \equiv 2 \pmod{3}$ is similar, with the relevant polynomials being
the factors of
$$\frac{t_6^0(\alpha,x)}{t_3^0(\alpha,x)} = 
\bigl(x-(\alpha+1)^2\bigr) 
\bigl(x^2+(1+2\alpha-2\alpha^2)x+(\alpha^2-\alpha+1)^2\bigr).$$
\end{proof}

We now have all of the tools required to prove Proposition~\ref{prop:main}.  We proceed using the conjugacy class representatives given in
Proposition~\ref{p:ccpn}.

\subsection{$\text{I}(\alpha)$}

This case is predictably straightforward:
$$g^k-1 = \mat{\alpha^k-1}{0}{0}{\alpha^k-1}.$$
By Lemma~\ref{lemma:val} we are precisely in the situation of  Corollary~\ref{cor:dcp} with $k_0 = \oab$ and
$$a=b=\voa.$$
Thus the double coset type in this case is $\dcp(\oab;\voa)$.

\subsection{$\text{I}'_{\mu,\nu}(\alpha,\beta)$}

Consider
$$g=\mat{\alpha}{\beta p^\nu}{p^\mu}{\alpha}$$
with $\alpha,\beta$ units and $1 \leq \mu < \nu < n$.  We apply the framework of Section~\ref{ss:np}:
$$g^k-1 = \mat{\rkm(\alpha,\beta p^\nu)-1}{\beta p^\nu\skm(\alpha,\beta p^\nu)}{p^\mu \skm(\alpha,\beta p^\nu)}{\rkm(\alpha,\beta p^\nu)-1}.$$
To compute the Smith normal form, we first must estimate the valuations of the entries above.  If $\oab \nmid k$, then $r-1$ is a unit and the Smith normal form is immediately seen to be the identity.
Assume therefore that $\oab \mid k$.
The smallest valuation terms in $r-1$ are the constant and linear terms:
$$\rkm(\alpha,p^m\beta)-1 = (\alpha^k-1) + \binom{k}{2}\alpha^{k-2}\beta p^{\mu+\nu}+\cdots$$
of valuation $\voa+v(k)$ and $v(k)+v(k-1)+\mu+\nu$ respectively.  Clearly $p^\mu s$ has smaller valuation than $\beta p^\nu s$; the smallest valuation term in $p^\mu s$ is always the constant term:
$$p^\mu \skm(\alpha,\beta p^\nu) = k\alpha^{k-1}p^\mu + \cdots$$
of valuation $v(k)+\mu$.

Assume first that $\voa \leq \mu$, so that $r-1$ is the entry of least valuation $\voa + v(k)$ (which is strictly less than the valuation of the linear term under the assumption).  
By (\ref{eq:smith}) the Smith normal form is therefore
$$\mat{p^{\voa+v(k)}}{0}{0}{p^{u_1+v(k)+\mu-\voa}}$$
using Corollary~\ref{cor:val} to evaluate the valuation of the determinant $t_k^\mu(\alpha,\beta p^\nu)$; here $u_1=v(\beta p^\nu -\zkm(\alpha))$.  
We apply Corollary~\ref{cor:dcp} to conclude that the double coset type is $\dcp(\oab;\voa,u_1+\mu-\voa)$ in this case.

If instead $\voa > \mu$, then $p^\mu s$ is the term of least valuation and instead we get:
$$\mat{p^{\mu+v(k)}}{0}{0}{p^{u_1+v(k)}}.$$
Thus the double coset type is $\dcp(\oab;\mu,u_1)$.

\subsection{$\text{I}'_\mu(\alpha)$}

Consider 
$$g = \mat{\alpha}{0}{p^\mu}{\alpha}.$$
Although it is overkill in this case, we might as well apply the approach of Section~\ref{ss:np}.
We obtain
$$g^k-1 = \mat{\rkm(\alpha,0)}{0 \skm(\alpha,0)}{p^\mu \skm(\alpha,0)}{\rkm(\alpha,0)} =
 \mat{\alpha^k-1}{0}{k\alpha^{k-1}p^\mu}{\alpha^k-1}.$$
If $\oab \nmid k$, then $\alpha^k-1$ is a unit and one immediately sees that the Smith normal form is the identity.  Assume therefore that $\oab \mid k$.
If $\voa \leq \mu$, then $\alpha^k-1$ has minimal valuation, so the Smith normal form is
$$\mat{p^{\voa+v(k)}}{0}{0}{p^{\voa+v(k)}}$$
and the double coset type is $\dcp(\oab;\voa)$.
If instead $\voa > \mu$, then $k \alpha^{k-1} p^\mu$ has minimal valuation, so that the Smith normal form is instead
$$\mat{p^{\mu+v(k)}}{0}{0}{p^{2\voa-\mu+v(k)}}$$
and the double coset type is $\dcp(\oab;\mu,2\voa-\mu)$.

\subsection{ $\text{I}^\pm_\mu(\alpha)$}

Consider
$$g=\mat{\alpha}{\beta p^\mu}{p^\mu}{\alpha}$$
so that
$$g^k-1 = \mat{\rkm(\alpha,\beta p^\mu )-1} {\beta p^\mu \skm(\alpha,p^\mu \beta)} {p^\mu \skm(\alpha,\beta p^\mu)}{\rkm(\alpha,\beta p^\mu )-1}.$$
As usual, the Smith normal form is the identity if $\oab \nmid k$, so we assume $\oab \mid k$.
The smallest valuation term for $r-1$ can be either $i=0$ or $i=1$, while for $s$ it is always $i=0$:
\begin{align*}
v\bigl(r_{k,\mu}(\alpha,\beta p^\mu)-1\bigr) &= \begin{cases} \voa+v(k) & \voa < v(k-1)+2\mu \\ v(k)+v(k-1)+2\mu &  \voa > v(k-1)+2\mu \end{cases} \\
v\bigl(p^\mu s_{k,\mu}(\alpha,\beta p^\mu)\bigr) &= v(k) + \mu.
\end{align*}

If $\voa < \mu$, then $r-1$ is the term of minimal valuation $\voa+v(k)$, so that the Smith normal form is
$$\mat{p^{\voa+v(k)}}{0}{0}{p^{u_2+\mu-\voa+v(k)}}$$
with $u_2 = v(\beta p^\mu - \zkm(\alpha))$.  Note that $v(\zkm(\alpha))=2\voa-\mu < \mu = v(\beta p^\mu)$, so that
$u_2=2\voa-\mu$.  Thus in this case the double coset type is $\dcp(\oab;\voa)$.

If $\voa > \mu$, then $p^\mu s$ has minimal valuation and the Smith normal form is
$$\mat{p^{\mu+v(k)}}{}{}{p^{u_2+v(k)}}.$$
This time we have $u_2 = \mu$, so that the double coset type is $\dcp(\oab;\mu)$.

If $\voa = \mu$, then either approach works but $u_2$ can no longer be determined purely from valuations, so that the double coset type is $\dcp(\oab;\mu,u_2)$.

\subsection{$\text{II}(\alpha,\beta)$} \label{ss:iiab}

We have 
$$g = \mat{\alpha}{\beta p}{1}{\alpha}$$
so that
$$g^k-1 = \mat{r_k^0(\alpha,\beta p)-1}{\beta p s_k^0(\alpha,\beta p)}{s_k^0(\alpha,\beta p)}{r_k^0(\alpha,\beta p)-1}.$$
As always, it suffices to consider the case $\oab \mid k$.

Except in the exceptional case $p=3$ and $v(\beta) = 0$,
the  term of least valuation appearing in either of $r-1$, $s$ is the $i=0$ term in $s$ of valuation $v(k)$.
Thus the Smith normal form is simply
$$\mat{p^{v(k)}}{}{}{p^{u_3+v(k)}}$$
where $u_3 = v(\beta p-z^0(\alpha))$ and $v(z^0(\alpha))=2\voa$.  The double coset type is therefore
$\dcp(\oab; 0, u_3)$.

For $p=3$ and $\beta \in \Z_3^\times$ we instead apply Lemma~\ref{l:val3}, which tells us that
for $3\oab \mid k$ we have
$$v\bigl(t_k^0(\alpha,3\beta)\bigr) = u_4 + 2v(k);$$ here
$$u_4 = \min\{2v(6\beta-2\alpha^2-2\varepsilon \alpha+1)-1,
2v(2\alpha+\varepsilon)\}$$
with $\varepsilon = 1$ (resp.\ $-1$) if $\alpha \equiv 1 \pmod{3}$ (resp.\ $\alpha \equiv 2 \pmod{3}$)

Note that the definition
$$t_k^0(\alpha,3\beta) = \bigl(r_k^0(\alpha,3\beta)-1\bigr)^2 - 3\beta s_k^0(\alpha,3\beta)^2$$
shows that if $u_4$ is even, then we must have
$$\frac{u_4}{2} = v\bigl(r_k^0(\alpha,3\beta)-1\bigr) < v\bigl(s_k^0(\alpha,3\beta)\bigr)$$
while if $u_4$ is odd, then we must have
$$\frac{u_4-1}{2} = v\bigl(s_k^0(\alpha,3\beta)\bigr) < v\bigl(r_k^0(\alpha,3\beta)-1\bigr).$$
In the former case, it follows that the Smith normal form is
$$\mat{p^{\frac{u_4}{2}+v(k)}}{0}{0}{p^{\frac{u_4}{2}+v(k)}},$$
resulting in a double coset type of $\dcp(\oab;\frac{u_4}{2})$
while in the latter case it is
$$\mat{p^{\frac{u_4-1}{2}+v(k)}}{0}{0}{p^{\frac{u_4+1}{2}+v(k)}}$$
with double coset type $\dcp(\oab;\frac{u_4-1}{2},\frac{u_4+1}{2})$.


\subsection{$\text{III}(\alpha,\beta)$} \label{ss:iii}

We have
$$g = \mat{\alpha}{0}{0}{\beta}$$
so that
$$g^k-1 = \mat{\alpha^k-1}{0}{0}{\beta^k-1}.$$
Set
\begin{align*}
\voa &= v(\alpha^{\oab}-1) \\
\vob &= v(\beta^{\obb}-1)
\end{align*}
and assume without loss of generality that $\voa \leq \vob$.
If $\oab=\obb$, then we are still in the situation of Corollary~\ref{cor:dcp}.  Otherwise we are in precisely the situation of Lemma~\ref{lemma:diag} with $k_1=\oab$, $k_2 = \obb$, $a = \voa$, $b= \vob$.

\subsection{$\text{IV}(\alpha,\beta)$}

Let $g \in G$ be a matrix with irreducible characteristic polynomial $f$ of unit discriminant.  Let $o_{\bar{g}}$ denote the order of the image of $g$ in $\GL_2(\Fp)$; we can write the
order $o_g$ of $g$ as $o_g = o_{\bar{g}}p^{v_g}$ for some $v_g \geq 0$.
Let $\O$ denote the ring of integers of the unramified quadratic extension of $\Qp$.  Then by Hensel's lemma $f$ has conjugate roots
$\gamma,\bar{\gamma}$ in ${\O}/p^n$ and $g$ is conjugate to
$$\mat{\gamma}{0}{0}{\bar{\gamma}}$$
in $\GL_2({\O}/p^n)$.  As $\gamma$ and $\bar{\gamma}$ both have order $o_g$ we may apply the simplest case of
$\text{III}$ to conclude that all orbits of $\left< g \right>$ on $W \otimes {\O}$ have order $o_g$.  As orbits of elements in $W \subseteq W \otimes {\O}$ will lie in $W$, it follows that all orbits in $W$ also have order $o_g$.  Thus the double coset type is simply $\dcp(o_{\bar{g}};v_g)$.\

\section{Double coset types: ramified cases} \label{s:ram}

\subsection{Multiplicative reduction}

We consider in this section double coset types  for the decomposition and inertia groups which occur for elliptic curves in certain ramified cases.
We begin with the case of multiplicative reduction.  Fix an odd prime power $p^n$,
$\alpha \in \Zpns$, $0 \leq b_{1} \leq b_{2} \leq n$ and $\vep \in \{ \pm 1\}$.  Consider subgroups
$$ I \subseteq D \subseteq \GL_2(\Zpn)$$
given by
\begin{align*}
D &= \left\{ \mat{\alpha^i}{p^{b_1} j}{0}{\vep^i} ; 0 \leq i < \oap, 0 \leq j < p^{n-b_1} \right\} \\
I &= \left\{ \mat{1}{p^{b_2} j}{0}{1} ; 0 \leq j < p^{n-b_2} \right\}.
\end{align*}
Here $\oap$ is the least common multiple of $\oa$ and the order $\oe$ of $\vep$; that is,
$\oap = 2\oa$ if $\vep = -1$ and $\oa$ is odd, and $\oap = \oa$ otherwise.  
We proceed to compute the double coset type of the pair $(D,I)$ as in Section~\ref{s:dcp}.

Although it is in principle appealing to try to develop a version of the arithmetic function theory we used for cyclic subgroups in this setting, the subgroup lattice of $D$ is far too complicated to make that practical.  We instead resort to a lengthy series of calculations for the actions of $D$ and $I$ on 
$$\GL_2(\Zpn)/\Gamma \cong W = \left\{ \vec{x}{y} ; x,y \in \Zpn, v(x)=0 \text{~or~} v(y)=0 \right\}.$$

As usual, let $\oab$ denote the order of the reduction $\bar{\alpha}$ and set 
$\voa = v(\alpha^{\oab}-1)$.  We also write $\oabp$ for the least common multiple of $\oab$
and $\oe$; note that $\oap = \oabp p^{n-\voa}$.

\begin{lemma} \label{l:stab}
Fix $w=\vec{x}{y} \in W$.  The stabilizer of $D$ acting on $w$ has order given by:
$$\St_{D,w} = \begin{cases} \frac{\oap}{\oa} p^{n-b_1} & y=0; \\
\frac{\oap}{\oe} p^{v(y)} & v(x) \geq b_1+v(y); \\
p^{\min\{n-b_1+v(x),n-\voa+v(y)\}} & v(x) < b_1+v(y), y \neq 0. \end{cases}$$
The stabilizer $\St_{I,w}$ of $I$ acting on $w$ has order $p^{\min\{v(y),n-b_2\}}$.
\end{lemma}
\begin{proof}
We begin by computing the stabilizer for $D$:
$$\mat{\alpha^i}{p^{b_1} j}{0}{\vep^i} \vec{x}{y} = \vec{x}{y}$$
if and only if
\begin{equation} \label{eq:vm}
(\alpha^i-1)x = -p^{b_1} j y
\end{equation}
and either $\vep=1$; or $\vep = -1$ and $i$ is even; or $y=0$.

Consider first the case $y=0$.  Then $x \in \Zpns$ so that by (\ref{eq:vm}) we must have
$\oa \mid i$.  Thus there are $\frac{\oap}{\oa}$ possible choices of $i$ and $p^{n-b_1}$ possible choices of $j$, resulting in a stabilizer of order $\frac{\oap}{\oa} p^{n-b_1}$.

Assume next that 
$$v(x) \geq b_1+v(y).$$
Then the equation (\ref{eq:vm}) can always be solved for $j$ regardless of $i$, and there will be precisely $p^{v(y)}$ such solutions.  If $\vep=-1$, then only the even $i$ give rise to elements of the stabilizer, so that in either case the stabilizer has order 
$\frac{\oap}{\oe} p^{v(y)}$.

Finally suppose that
$$v(x) < b_1+v(y).$$
Then (\ref{eq:vm}) only has solutions if 
$$v(\alpha^i-1) \geq b_1 + v(y)-v(x).$$
This requires $\oab \mid i$ and
$$ v(i) \geq b_1 -\voa+v(y)-v(x).$$
Since we in fact need $\oabp \mid i$, we see that the stabilizer has order
$$\frac{\oap}{\oabp p^{\max\{b_1-\voa+v(y)-v(x),0\}}} =
p^{\min\{n-b_1-v(y)+v(x),n-\voa\}}$$
since $\oap = \oabp p^{n-\voa}.$  As there are still $p^{v(y)}$ choices of $j$ for each such $i$, we conclude that the stabilizer has order
$$p^{\min\{n-b_1+v(x),n-\voa+v(y)\}}.$$

We turn now to the stabilizer for $I$, which is much simpler.  Indeed, 
$\mat{1}{p^{b_2} j}{0}{1}$ stabilizes $w$ if and only if $p^{b_2} j y=0$, which
occurs if and only if $v(j) \geq n-b_2-v(y)$.  There are therefore $p^{\min\{v(y),n-b_2\}}$ such $j$.
\end{proof}

\begin{cor} \label{c:orbit}
The orbit $D \cdot w$ of $w = \vec{x}{y} \in W$ has order
$$\# D \cdot w = \begin{cases} \oa & y = 0; \\
\oe p^{n-b_1-v(y)} & v(x) \geq b_1+v(y); \\
\oabp p^{\max\{n-v_o-v(x),n-b_1-v(y)\}} & v(x) < b_1+v(y), y \neq 0.
\end{cases}$$
The orbit $I \cdot w$ has order $p^{\max\{n-b_2-v(y),0\}}$.
\end{cor}

It is tedious, but not difficult, to convert Corollary~\ref{c:orbit} into the desired double coset types.

\begin{prop} \label{prop:dcpram}
The double coset type of the pair $(D,I)$ is given as follows:
\begin{description}
\item[$b_1=0$] 
\begin{multline*}
\frac{p-1}{\oab}p^{\voa-1} \times (\oa,1) +
\sum_{u=\voa+1}^{n-1} \frac{(p-1)^2}{\oabp} p^{n+\voa-u-2} \times (\oabp p^{n-\voa},p^{\max\{n-b_2-u,0\}}) + \\
\sum_{u=1}^{\min\{\voa,n-1\}} \frac{(p-1)^2}{\oabp}p^{n-2} \times (\oabp p^{n-u},p^{\max\{n-b_2-u,0\}})
 +
\frac{p-1}{\oe} p^{n-1} \times (\oe p^n,p^{n-b_2})
\end{multline*} \vspace{.2cm}
\item[$0 < b_1 < \voa$]
\begin{multline*}
\frac{p-1}{\oab}p^{\voa-1} \times (\oa,1) +
\frac{p-1}{\oabp}p^{n-1}(p^{b_1}-1) \times (\oabp p^{n-b_1},p^{n-b_2}) + \\
\frac{p-1}{\oe}p^{n-1} \times (\oe p^{n-b_1},p^{n-b_2}) + \\
\sum_{u=1}^{\voa-b_1-1} \frac{(p-1)2}{\oabp} p^{n+b_1-2} \times (\oabp p^{n-b_1-u},p^{\max\{n-b_2-u,0\}}) + \\
\sum_{u=\voa-b_1}^{n-b_1-1} \frac{(p-1)^2}{\oabp} p^{n+\voa-u-2} \times (\oabp p^{n-\voa},p^{\max\{n-b_2-u,0\}}) + \\
\frac{p-1}{\oabp} p^{\voa-1}(p^{b_1}-1) \times (\oabp p^{n-\voa},1)
\end{multline*} \vspace{.2cm}

\item[$b_1 \geq \voa$]
\begin{multline*}
\frac{p-1}{\oab}p^{\voa-1} \times (\oa,1) + 
\sum_{u=0}^{n-b_2-1} \frac{(p-1)^2}{\oabp}p^{n+\voa-u-2} \times (\oabp p^{n-\voa}, p^{n-b_2-u}) + \\
\frac{p-1}{\oabp}p^{\voa-1}(p^{b_2}-1) \times (\oabp p^{n-\voa}, 1) + \\
\sum_{u=1}^{b_1-\voa} \frac{(p-1)^2}{\oabp}p^{n+\voa-2} \times (\oabp p^{n-\voa-u},p^{n-b_2}) + \\
\frac{p-1}{\oabp}p^{n-1}(p^{\voa-1}-1)  \times (\oabp p^{n-b_1},p^{n-b_2}) + 
\frac{p-1}{\oe}p^{n-1} \times (\oe p^{n-b_1},p^{n-b_2})
\end{multline*}  

\end{description}
\end{prop}

We write these double coset types as $\dcpm(\alpha,\varepsilon,b_1,b_2)$.

\begin{proof}
Consider first the subset of $W$ with $y=0$, which has order $(p-1)p^{n-1}$.  By Corollary~\ref{c:orbit} this divides into $D$-double cosets of order $\oa$ and $I$-double cosets of order $1$.  Thus the contribution to the double coset type from these elements of $W$ is
\begin{equation} \label{eq:dcpp1}
\frac{(p-1)p^{n-1}}{\oa} \times (\oa,1) = \frac{p-1}{\oab}p^{\voa-1} \times (\oa,1).
\end{equation}

For the remainder of the proof it is most convenient to consider three cases: $b_1=0$; $b_1 \geq \voa$; and $0 < b_1 < \voa$.  Consider first $b_1=0$.  If $v(y)=0$, then the $D$-orbits have order $\oe p^{n}$ and the $I$-orbits have order $p^{n-b_2}$.  There are $(p-1)p^{2n-1}$
such elements in $W$, yielding a contribution of
\begin{equation} \label{eq:dcpp2}
\frac{p-1}{\oe} p^{n-1} \times (\oe p^n,p^{n-b_2})
\end{equation}
from these elements.  The remaining elements have $v(x)=0$ and $1 \leq v(y) \leq n-1$: the $D$-orbits have order $\oabp p^{\max\{n-\voa,n-v(y)\}}$ and the $I$-orbits have order
$p^{\max\{n-b_2-v(y),0\}}$.  This yields a double coset type contribution of
\begin{multline} \label{eq:dcpp3}
\sum_{u=1}^{\min\{\voa,n-1\}} \frac{(p-1)^2}{\oabp}p^{n-2} \times (\oabp p^{n-u},p^{\max\{n-b_2-u,0\}}) + \\
\sum_{u=\voa+1}^{n-1} \frac{(p-1)^2}{\oabp} p^{n+\voa-u-2} \times (\oabp p^{n-\voa},p^{\max\{n-b_2-u,0\}}).
\end{multline}
Combining (\ref{eq:dcpp1}), (\ref{eq:dcpp2}) and (\ref{eq:dcpp3}) yields the full double coset type in this case.

Consider next the case $b_1 \geq \voa$.  If $v(x)=0$ and $0 \leq v(y) \leq n-1$, then the $D$-orbits have order
$\oabp p^{n-\voa}$.  If $0 \leq v(y) \leq n-b_2$, then the $I$-orbits have order $p^{n-b_2-v(y)}$, while if $n-b_2 \leq v(y) \leq n-1$ then they have order $1$.  This results in a contribution of
\begin{multline} \label{eq:dcpp4}
\sum_{u=0}^{n-b_2-1} \frac{(p-1)^2}{\oabp}p^{n+\voa-u-2} \times (\oabp p^{n-\voa}, p^{n-b_2-u}) + \\
\frac{p-1}{\oabp}p^{\voa-1}(p^{b_2}-1) \times (\oabp p^{n-\voa}, 1).
\end{multline}

If $v(y)=0$ and $1 \leq v(x) \leq b_1-\voa$, then the $D$-orbits have order $\oabp p^{n-\voa-v(x)}$ and
the $I$-orbits have order $p^{n-b_2}$.  If  $b_1-\voa+1 \leq v(x) \leq b_1-1$, then the $D$-orbits have order
$\oabp p^{n-b_1}$ and the $I$-orbits have order $p^{n-b_2}$.  Finally, if $b_1 \leq v(x) \leq n$, then
the $D$-orbits have order $\oe p^{n-b_1}$ and the $I$-orbits have order $p^{n-b_2}$.  The overall contribution to the double coset type is thus
\begin{multline} \label{eq:dcpp5}
\sum_{u=1}^{b_1-\voa} \frac{(p-1)^2}{\oabp}p^{n+\voa-2} \times (\oabp p^{n-\voa-u},p^{n-b_2}) + \\
\frac{p-1}{\oabp}p^{n-1}(p^{\voa-1}-1)  \times (\oabp p^{n-b_1},p^{n-b_2}) + 
\frac{p-1}{\oe}p^{n-1} \times (\oe p^{n-b_1},p^{n-b_2}).
\end{multline}
Combining (\ref{eq:dcpp1}), (\ref{eq:dcpp4}) and (\ref{eq:dcpp5}) yields the double coset type in this case.

It remains to consider the case $0 < b_1 < \voa$.  If $v(y)=0$ and $0 \leq v(x) \leq b_1-1$, then
the $D$-orbits have order $\oabp p^{n-b_1}$ and the $I$-orbits have order $p^{n-b_2}$, while
if $b_1 \leq v(x) \leq n$, the $D$-orbits have order $\oe p^{n-b_1}$ and the $I$-orbits have order $p^{n-b_2}$, resulting in a double coset type contribution of
\begin{equation} \label{eq:dcpp6}
\frac{p-1}{\oabp}p^{n-1}(p^{b_1}-1) \times (\oabp p^{n-b_1},p^{n-b_2}) +
\frac{p-1}{\oe}p^{n-1} \times (\oe p^{n-b_1},p^{n-b_2}).
\end{equation}

If $v(x)=0$ and $1 \leq v(y) < \voa-b_1$, then the $D$-orbits have order $\oabp p^{n-b_1-v(y)}$ and
the $I$-orbits have order $p^{\max\{n-b_2-v(y),0\}}$.  If $\voa -b_1 \leq v(y) < n-b_1$, then the $D$-orbits have order $\oabp p^{n-\voa}$ and the $I$-orbits have order $p^{n-b_2-v(y)}$.  Finally, if
$n-b_1 \leq v(y) \leq n-1$, then the $D$-orbits have order $\oabp p^{n-\voa}$ and the $I$-orbits have order $1$:
\begin{multline} \label{eq:dcpp7}
\sum_{u=1}^{\voa-b_1-1} \frac{(p-1)2}{\oabp} p^{n+b_1-2} \times (\oabp p^{n-b_1-u},p^{\max\{n-b_2-u,0\}}) + \\
\sum_{u=\voa-b_1}^{n-b_1-1} \frac{(p-1)^2}{\oabp} p^{n+\voa-u-2} \times (\oabp p^{n-\voa},p^{\max\{n-b_2-u,0\}}) + \\
\frac{p-1}{\oabp} p^{\voa-1}(p^{b_1}-1) \times (\oabp p^{n-\voa},1)
\end{multline}
Combining (\ref{eq:dcpp1}), (\ref{eq:dcpp6}) and (\ref{eq:dcpp7}) mercifully completes the proof.
\end{proof}

\subsection{Good ordinary primes}

For this section fix $\alpha \in \Zpns$ and consider the pair $(D,I)$ given by
\begin{align*}
D &= \left\{ \mat{\alpha^{k}a}{b}{0}{\alpha^{-k}} ; 0 \leq k < \oa, a \in \Zpns, b \in \Zp  \right\} \\
I &= \left\{ \mat{a}{b}{0}{1} ; a \in \Zpns, b \in \Zpn \right\}.
\end{align*}
We compute the double coset type of the pair $(D,I)$ using the methods of the previous section.

\begin{lemma} \label{l:stord}
For $w = \vec{x}{y} \in W$, we have
\begin{align*}
\# \St_{D,w} &= \begin{cases} (p-1)p^{n-1} & v(y) = 0 \\
p^{\min\{n+v(y),2n-\voa\}} & v(x)=0, \quad 0 < v(y) < n \\
\oab p^{2n-\voa} & v(x)=0, \quad y=0
\end{cases} \\
\# \St_{I,w} &= \begin{cases} (p-1)p^{n-1} & v(y) = 0 \\
p^n & v(x)=0, \quad v(y) > 0.\end{cases}
\end{align*}
\end{lemma}
\begin{proof}
We begin with $I$ which is simpler.  Consider the equation
$$\mat{a}{b}{0}{1} \vec{x}{y} = \vec{x}{y}$$
or equivalently
$$(a-1)x = -by.$$
Note that if $v(y)=0$, then one can solve for $b$ in terms of $a$, so that the stabilizer has
order $(p-1)p^{n-1}$.  If $v(x)=0$ and $v(y)>0$, then one can instead solve for $a$ in terms of $b$ and the stabilizer has order $p^n$; since $v(y)>0$, the resulting value of $a$ is guaranteed to be a unit.

Turning to $D$,
a matrix $\mat{\alpha^k a}{b}{0}{\alpha^{-k}}$ stabilizes $w=\vec{x}{y} \in W$ if and only
\begin{align*}
(\alpha^k a-1) x &= -by \\
\alpha^{-k}y &= y.
\end{align*}
If $v(y)=0$, then the second equation forces $k=0$ while we can solve the first for $b$ in terms of $a$.  It follows that $\St_{w}$ has order $(p-1)p^{n-1}$ as before.  At the other extreme, if $y=0$
then the second equation is vacuous and the first becomes $a=\alpha^{-k}$.  Since $b$ is arbitrary, there are $\oa p^{n} = \oab p^{2n-\voa}$ solutions.

If $v(x)=0$ and $0 < v(y) < n$, then the second equation forces $\oab \mid k$ and
$$v(k) \geq n -\voa-v(y).$$
Thus there are $p^{\min\{v(y),n-\voa\}}$ choices for $k$.  The first equation can be solved for $a$ in terms of $b$ and $\alpha^k$, so that in total the stabilizer has order
$p^{\min\{n+v(y),2n-\voa\}}.$
\end{proof}

\begin{cor}
The orbits of $w \in W$ have order
\begin{align*}
\# D \cdot w &= \begin{cases} \oab p^{2n-\voa} & v(y) =0 \\
\oab(p-1) p^{\max\{2n-1-\voa-v(y),n-1\}} & v(x) = 0, \quad 0 < v(y) < n \\
p^{n-1} & v(x)=0, \quad y=0
\end{cases}\\
\# I \cdot w &= \begin{cases} p^n & v(y) = 0 \\
(p-1)p^{n-1} & v(x)=0, \quad v(y)>0 \end{cases}
\end{align*}
\end{cor}

\begin{prop} \label{prop:dcpord}
The double coset type of the pair $(D,I)$ is
\begin{multline*}
\frac{p-1}{\oab} p^{\voa-1} \times \oab (p^{2n-\voa},p^n) + \\
\sum_{u=1}^{n-\voa-1}  \frac{p-1}{\oab}p^{\voa-1}    \times (\oab(p-1)p^{2n-1-\voa-u},(p-1)p^{n-1}) + \\
\frac{p^{\min\{n-1,\voa\}}-1}{\oab} \times (\oab(p-1)p^{n-1},(p-1)p^{n-1}) +
1 \times ((p-1)p^{n-1},(p-1)p^{n-1})).
\end{multline*}
\end{prop}

We write this double coset type as $\dcpo(\alpha).$

\begin{proof}
There are $(p-1)p^{2n-1}$ elements of $W$ with $v(y)=0$, all lying $D$-orbits of order
$\oab p^{2n-\voa}$ and $I$-orbits of order $p^n$:
$$\frac{p-1}{\oab} p^{\voa-1} \times \oab (p^{2n-\voa},p^n).$$
When $v(x)=0$ and each $0 < v(y) < n-\voa$, there are
$(p-1)^2p^{2n-2-v(y)}$ elements lying in $D$-orbits of order $\oab (p-1)p^{2n-1-\voa-v(y)}$ and
$I$-orbits of order $(p-1)p^{n-1}$:
$$\sum_{u=1}^{n-\voa-1}  \frac{p-1}{\oab}p^{\voa-1}    \times (\oab(p-1)p^{2n-1-\voa-u},(p-1)p^{n-1}).$$
The $(p-1)p^{n-1}(p^{\min\{n-1,\voa\}}-1)$ elements with $\max\{1,n-\voa\} \leq v(y) \leq n-1$ lie in $D$-orbits of order
$\oab (p-1)p^{n-1}$ and $I$-orbits of order $(p-1)p^{n-1}$:
$$\frac{p^{\min\{n-1,\voa\}}-1}{\oab} \times (\oab(p-1)p^{n-1},(p-1)p^{n-1}).$$
Finally, the $(p-1)p^{n-1}$ elements with $y=0$ lie in $D$-orbits of order $(p-1)p^{n-1}$ and $I$-orbits of order $(p-1)p^{n-1}$:
$$1 \times ((p-1)p^{n-1},(p-1)p^{n-1}).$$
\end{proof}

\section{Distribution of factorization types} \label{sec:dist}

Fix an odd integer $N=p_1^{n_1} \cdots p_m^{n_m}$ and let $L/E$ be a Galois extension of number fields with $\Gal(L/E) \cong \GL_2(\Z/N)$.  Let $\Gamma$ denote the subgroup $\left\{ \mat{1}{*}{0}{*} \right\}$ of $G$ and set $K=L^\Gamma$.  We may apply our results to determine the factorization types which occur for the extension $K/E$ as well as the density of each.

Let $\C$ be a conjugacy class in $\GL_2(\Z/N)$.  There are conjugacy classes $\C_i \subseteq \GL_2(\Z/p_i^{n_i})$ with representatives $g_i$ as in Table~\ref{t:cc} such that
$$\C = \C_1 \times \cdots \times \C_m.$$
In particular, this allows one to compute $\#\C$.  Using Tables~\ref{t:dcp2} and~\ref{t:dcp}, one
can also compute the double coset type $\dcp_N(\C)$ of any $g \in \C$ as
$$\dcp_N(\C) = \dcp_{p_1^{n_1}}(g_1)
\otimes \cdots \otimes \dcp_{p_m^{n_m}}(g_m).$$
Here of course the indices refer to the group in which the double coset types are computed.

\begin{prop}
Every prime in $E$, unramified in $L/E$, has factorization type in $K/E$ equal to $\dcp(\C)$ for some 
conjugacy class $\C \subseteq \GL_2(\Z/N)$.
If $\lambda$ is a factorization type which occurs for $K/E$, then the density of primes of $E$ with factorization type $\lambda$ is
$$\frac{1}{[L:E ]} \cdot \sum_{\dcp(\C)=\lambda} \#\C.$$
\end{prop}
\begin{proof}
This is immediate from
Propositions~\ref{prop:main} and~\ref{p:ccpn} combined with the Chebotarev density theorem.
\end{proof}

In practice, of course, one computes proportions of factorization types by enumerating over all conjugacy classes.

\begin{example}
Table~\ref{t:63} gives the proportion of each factorization type in the case $N=63$.
We illustrate the sort of computation which, repeated many times, goes into the creation of the table.  Consider the conjugacy class $\C$ in $\GL_2(\Z/63)$ which is the product of the conjugacy class $\text{I}^-_1(2,2)_9$ of $\GL_2(\Z/9)$ and $\text{III}(2,6)_7$ of $\GL_2(\Z/7)$.
The former is represented by the matrix $\mat{2}{6}{3}{2} \in \GL_2(\Z/9)$ while the latter is represented by
the matrix $\mat{2}{0}{0}{6} \in \GL_2(\Z/7)$.  Applying the Chinese remainder theorem, $\C$ is thus represented by $\mat{2}{42}{21}{20}$.

The conjugacy class $\text{I}^-_1(2,2)_9$ has order 8.  As
$\sigma_{\bar{2}}=2$ and $v_{2}=2$,
by Proposition~\ref{prop:main}, it has double coset type $\dcp(2;1,u_2)$ where
$$u_2=v\bigl(6-z^1(2)\bigr).$$
We compute that $z^1(2)=3$, so that $u_2=1$.  Thus the double coset type is simply
$\dcp_9(2;1,1) = 12 \times 6$.

The conjugacy class $\text{III}(2,6)_7$ has order 56.  We have
$\sigma_{\bar{2}}=3$, $v_{2} = 1$, $\sigma_{\bar{6}}=2$, $v_6=1$.  Thus the double coset type
is $\dcp_7(3,2;1) = 3 \times 2 + 2 \times 3 + 6 \times 6$.

We conclude that
$$\dcp_{63}(\C) = (12 \times 6)_{9} \otimes (3 \times 2 + 2 \times 3 + 6 \times 6)_{7} =
(576 \times 6)_{63}.$$
This conjugacy class contributes $8 \cdot 56 = 448$ to the count for that double coset type.  Repeating such calculations over all conjugacy classes yields 
Table~\ref{t:63}.

The least common decomposition type, of course, is splitting completely, with only $\frac{1}{7114162}$ of rational primes exhibiting it.  The most common is
decomposing as a product of 144 primes of degree 24, which occurs for about 22.25\% of rational primes.  The most distinct residual degrees in any factorization is six, which occurs for five different types.  The largest residual degree to occur is 168, which occurs for 3.79\% of primes.
The most interesting factorization type, perhaps, is
$$18 \times 2 + 12 \times 6 + 18 \times 14 + 18 \times 18 + 12 \times 42 + 18 \times 126 $$
which occurs for 0.44\% of primes.  It has the maximal six distinct residual degrees, ranging from 2 to 126.  It is the double coset type of the nine conjugacy classes represented by the matrices
$$\left\{ \mat{\alpha}{0}{1}{\alpha} ; \alpha \in \{8,13,20,29,34,41,50,55,62\} \right\}.$$
\end{example}

\begin{table}  \tiny
$$\begin{array}{lr}
\text{Type} & \text{Number} \\ \hline 
3456 \times 1  &  1 \\
432 \times 1 + 1512 \times 2  &  56 \\
108 \times 1 + 378 \times 2 + 108 \times 3 + 378 \times 6  &  1792 \\
36 \times 1 + 126 \times 2 + 24 \times 3 + 84 \times 6 + 36 \times 9 + 126 \times 18  &  12096 \\
36 \times 1 + 198 \times 2 + 504 \times 6  &  12096 \\
36 \times 1 + 36 \times 2 + 54 \times 6 + 36 \times 7 + 36 \times 14 + 54 \times 42  &  10368 \\
36 \times 1 + 198 \times 2 + 36 \times 7 + 198 \times 14  &  5184 \\
864 \times 1 + 864 \times 3  &  32 \\
108 \times 1 + 108 \times 3 + 504 \times 6  &  3584 \\
36 \times 1 + 24 \times 3 + 126 \times 6 + 36 \times 9 + 126 \times 18  &  24192 \\
36 \times 1 + 24 \times 3 + 36 \times 7 + 36 \times 9 + 24 \times 21 + 36 \times 63  &  10368 \\
108 \times 1 + 108 \times 3 + 108 \times 7 + 108 \times 21  &  1536 \\
288 \times 1 + 192 \times 3 + 288 \times 9  &  216 \\
432 \times 1 + 504 \times 6  &  112 \\
432 \times 1 + 432 \times 7  &  48 \\
1728 \times 2  &  167 \\
432 \times 2 + 12 \times 3 + 426 \times 6  &  12096 \\
18 \times 2 + 36 \times 3 + 120 \times 6 + 36 \times 9 + 126 \times 18  &  24192 \\
54 \times 2 + 12 \times 3 + 48 \times 6 + 54 \times 14 + 12 \times 21 + 48 \times 42  &  10368 \\
216 \times 2 + 504 \times 6  &  12656 \\
18 \times 2 + 12 \times 6 + 18 \times 14 + 18 \times 18 + 12 \times 42 + 18 \times 126  &  31104 \\
54 \times 2 + 54 \times 6 + 54 \times 14 + 54 \times 42  &  25344 \\
144 \times 2 + 96 \times 6 + 144 \times 18  &  12744 \\
216 \times 2 + 216 \times 14  &  5328 \\
1152 \times 3  &  47058 \\
12 \times 3 + 570 \times 6  &  338688 \\
36 \times 3 + 126 \times 6 + 36 \times 9 + 126 \times 18  &  157248 \\
12 \times 3 + 66 \times 6 + 12 \times 21 + 66 \times 42  &  114048 \\
288 \times 3 + 288 \times 9  &  49464 \\
36 \times 3 + 36 \times 9 + 36 \times 21 + 36 \times 63  &  51840 \\
144 \times 3 + 144 \times 21  &  41184 \\
864 \times 4  &  10020 \\
108 \times 4 + 252 \times 12  &  24192 \\
72 \times 4 + 48 \times 12 + 72 \times 36  &  18144 \\
108 \times 4 + 108 \times 28  &  5184 \\
576 \times 6  &  386814 \\
144 \times 6 + 144 \times 18  &  305640 \\
18 \times 6 + 18 \times 18 + 18 \times 42 + 18 \times 126  &  155520 \\
72 \times 6 + 72 \times 42  &  237600 \\
432 \times 8  &  33648 \\
108 \times 8 + 108 \times 24  &  41664 \\
36 \times 8 + 24 \times 24 + 36 \times 72  &  36288 \\
54 \times 8 + 54 \times 56  &  10368 \\
288 \times 12  &  614664 \\
72 \times 12 + 72 \times 36  &  90720 \\
36 \times 12 + 36 \times 84  &  134784 \\
216 \times 16  &  45696 \\
54 \times 16 + 54 \times 48  &  83328 \\
18 \times 16 + 12 \times 48 + 18 \times 144  &  72576 \\
144 \times 24  &  1583136 \\
36 \times 24 + 36 \times 72  &  181440 \\
18 \times 24 + 18 \times 168  &  269568 \\
72 \times 48  &  1395072 \\
18 \times 48 + 18 \times 144  &  362880 \\
\end{array}$$
\caption{Factorization Types for $N=63$} \label{t:63}
\end{table}

Of course, the field $K$ is not defined sufficiently precisely to compute the double coset type for any specific prime $q$.  For this we need more information on the field $L$, which is the topic of the next two sections.

\section{Computing Frobenius}

Let $E$ be an elliptic curve over $\Q$ of conductor $N_E$ and fix an integer $N$.  For a prime $q \nmid NN_E$ we wish to determine the matrix (well defined up to conjugacy) in $\GL_2(\Z/N)$ of a Frobenius element at $q$.  
We use the approach of \cite{DukeToth} which yields a Frobenius matrix
\begin{equation} \label{frobmat}
\mat{\frac{a_q+b_q\delta_q}{2}}{b_q}{\frac{b_q(\Delta_q-\delta_q)}{4}}{\frac{a_q-b_q\delta_q}{2}}
\end{equation}
independent of $N$.  Here
\begin{align*}
a_q &= q+1 - \#E(\Fq) \\
\Delta_q &= \disc \End(E_{\Fq}) \cap \Q(\pi_{E_{\Fq}}) \\
b_q &= \sqrt{\frac{a_q^2-4q}{\Delta_q}} > 0\\
\delta_q &= \begin{cases} 0 & \Delta_q \equiv 0 \pmod{4} \\ 1 & \Delta_q \equiv 1 \pmod{4}
\end{cases}
\end{align*}
Here $\End(E_{\Fq})$ denotes the endomorphism ring
of $E$ over $\Fq$ while $\pi_{E_{\Fq}} : E_{\Fq} \to E_{\Fq}$ denotes the
Frobenius endomorphism.

The Fourier coefficient $a_q$ is of course efficiently computable using the algorithm of Schoof \cite{Schoof} and its refinement \cite{Schoof2} by Atkin and Elkies.   To compute
$\Delta_q$, recall that $\End(E_{\Fq})$ is an order in a CM field (resp.\ quaternion algebra) when
$E$ is ordinary (resp.\ supersingular) at $q$.
The field $\Q(\pi_{E_{\Fq}})$ generated by Frobenius is
either $\Q$ or imaginary quadratic.  In the former case,
the intersection $\End(E_{\Fq}) \cap \Q(\pi_{E_{\Fq}})$
is $\Z$ and we set $\Delta_q=1$.  Otherwise the intersection is an order in the imaginary quadratic field $\Q(\pi_{E_{\Fq}})$ and $\Delta_q$ denotes its discriminant.  

$\Delta_q$ is necessarily divisible by the discriminant of the maximal order of $\Q(\pi_{E_{\Fq}})$.
On the other hand,
$\End(E_{\Fq}) \cap \Q(\pi_{E_q})$ contains
$\Z[\pi_{E_{\Fq}}]$, so that $\Delta_q$ must divide
the discriminant $a_q^2-4q$ of that ring.  One can
determine the precise value of $\Delta_q$ via several different algorithms; see
\cite[Section 4.2]{Kohel},
\cite[Section 5.2]{Sutherland}, \cite{BissonSutherland} and
\cite{Bisson}.  At the present time, we have implemented only the isogeny climbing portion of
\cite[Algorithm 1]{BissonSutherland}, which in Magma only works so long as $N$ is divisible only by primes $\le 59$.  (We note that if $E$ is supersingular at $q$, then $\End(E_{\Fq})$ is necessarily maximal so that $\Delta_q$ equals the discriminant of $\Q(\pi_{E_{\Fq}})$.

\begin{example}
Take $E=X_0(11)$ and $q=8689$.  One finds that $a_q=90$.  We compute
$$a_q^2-4q = -26656 = 14^2 \cdot -136.$$
Thus
$$\Z[14\sqrt{-34}] \subseteq \End(E_{\F_q}) \subseteq \Z[\sqrt{-34}].$$
Isogeny climbing shows that in fact 
$$\End(E_{\F_q}) = \Z[2\sqrt{-34}]$$
so that $\Delta_q = -544$.
Thus $b_q=7$ and $\delta_q=0$, resulting in a Frobenius matrix of
$$\mat{45}{7}{-952}{45}.$$
\end{example}

Assuming $E$ is fixed in the discussion,
we write $\Frob_{q}$ for the conjugacy class of (\ref{frobmat}) in $\GL_2(\Z/N)$ for any $q$ and $N$
with $q \nmid N_E N$.

\section{Dedekind zeta functions} \label{s:dzf}

Recall that the {\it Dedekind zeta function} of a number field $K$ is the Dirichlet series
$$\zeta_K(s) = \sum z_n n^{-s}$$
where
$z_n$ is the number of ideals of $\O_K$ of absolute norm $n$.  The results to this point enable us to compute $z_n$ for the number fields we have been
considering.

Fixing an elliptic curve $E$ and an odd integer $N$,
we write
$$\dcp_N(\Frob_{q})$$
for the double coset type of the conjugacy class of Frobenius at $q$ acting on $E[N]$.  Recall that if $E$ has full Galois image modulo $N=p_1^{n_1}\cdots p_m^{n_m}$, then by Lemma~\ref{lemma:dcpprod} and Proposition~\ref{prop:main} we have that
$$\dcp_N(\Frob_{q}) = \dcp_{p_1^{n_1}}(\Frob_{q}) \otimes \cdots \otimes \dcp_{p_m^{n_m}}(\Frob_{q}).$$

\begin{thm} \label{thm:main2}
Let $E$ be an elliptic curve over $\Q$ and fix an odd $N=p_1^{n_1}\cdots p_m^{n_m}$ such that $E$ has full Galois image modulo $N$.  Let $f_E \in S_2(N_E)$ be the cuspidal newform corresponding
to $E$ with $N_E$ the conductor of $E$.  Let $$\Gamma = \left\{ \mat{1}{b}{0}{a} ; a \in \ZNs, b \in \ZN \right\}$$
and set $K=\Q(E[N])^\Gamma$.  Assume further that:
\begin{enumerate}
\item $E$ is semistable (that is, $N_E$ is squarefree);
\item $E$ has good ordinary reduction at each $p_i$ and there does not exist a companion form to $f_E$ in $S_{p_i-1}(N_E,\F_{p_i})$.
\end{enumerate}
Then the factorization type of a prime $q$ in the ring of integers of $K$ is given as follows.
\begin{description}
\item[$q \nmid NN_E$] $\dcp_N(\Frob_{q})$
\item[$q \mid N_E$] $\dcpm_{p_1^{n_1}}(q\varepsilon,\varepsilon,b_{1,1},b_{2,1}) \otimes \cdots \otimes \dcp_{p_m^{n_m}}(q\varepsilon,\varepsilon,b_{1,m},b_{2,m})$
where:
\begin{itemize}
\item $\varepsilon=1$ (resp.\ $-1$) if $E$ has split (resp.\ non-split) multiplicative reduction at $q$;
\item $p^{b_{1,i}}$ equals the order of the quotient of $\Q_{q}^\times/\Q_q^{\times p_i^{n_i}}$ by the
multiplicative Tate period $\vartheta \in \Q_q^\times$; 
\item $b_{2,i}$ equals $-v_{p_i}(v_q(j_E))$ with $j_E$ the $j$-invariant of $E$.
\end{itemize}
\item[$p_i \mid N$] $\dcp_{\frac{N}{p_i^{n_i}}}(\Frob_{q}) \otimes \dcpo_{p_i^{n_i}}(\alpha_i)$
where $\alpha_i \in (\Z/p_i^{n_i})^\times$ is the unit root of $x^2-a_{p_i}x+p_i$.
\end{description}
\end{thm}

\begin{rmk} So long as all prime divisors of $N$ are at most 59, every double coset type
in Theorem~\ref{thm:main2} can be computed readily, allowing for the computation of the Dedekind zeta function for large non-Galois number fields.
\end{rmk}

\begin{proof} Given Lemma~\ref{lemma:factor}, all that we need check is that the decomposition and inertia groups in the case of multiplicative reduction and good ordinary reduction (without a companion form) are of the form given in Section~\ref{s:ram}.  Assume first that $E$ has split multiplicative reduction at $q$ and fix
some prime power divisor $p^n$ of $N$.  Then using the parameterization of $E_{\Q_q}$ by the Tate curve, we have that
$$D = \Gal\bigl(\Q_q(\zeta_{p^n},\vartheta^{\frac{1}{p^n}})/\Q_q\bigr)$$
with $\zeta_{p^n}$ a primitive $p^{n\text{th}}$ root of unity and $\vartheta \in \Q_q^\times$ the multiplicative Tate period of $E$ over $\Q_q$.
This Galois group can be realized as a subgroup of the semi-direct product
$$\Zpn \rtimes \Zpns \cong \left\{ \mat{a}{b}{0}{1} ; a \in \Zpns, b \in \Zpn \right\} \subseteq \GL_2(\Zpn).$$
The image of $D$ under projection to $\Zpns$ must be $\left< q \right>$ (corresponding to Frobenius), while the subgroup of
$\Zpn$ has order equal to the order of the subgroup of $\Q_q^\times/\Q_q^{\times p^n}$ generated by $\vartheta$.  This confirms the description of 
$D$ in this case.

Since $\Q_q(\zeta_{p^n})/\Q_q$ is unramified, the inertia group $I$ is also the inertia group of
$$\Gal\bigl(\Q_q(\zeta_{p^n},\vartheta^{\frac{1}{p^n}})/\Q_q(\zeta_{p^n})\bigr).$$
By Kummer theory, this corresponds to the image of $\vartheta$ under the valuation map
$$\Q_q^\times/\Q_q^{\times p^n} \cong \F_q^\times/\F_q^{\times p^n} \times \Z/p^n \to \Z/p^n.$$
As $v(\vartheta) = -v(j_E)$, this confirms the description of $I$ in this case.

The non-split multiplicative case is similar, with the matrix of Frobenius now given by $\mat{-q}{*}{0}{-1}$.

When $E$ has good ordinary reduction at $p$, the integral $p$-adic Galois representation is
given by
$$\rho_{E,p}|_{G_{\Q_p}} = \mat{\chi\psi}{\nu}{0}{\psi^{-1}}$$
where $\chi$ is the cyclotomic character; $\psi$ is the unramified character sending Frobenius to
the unit root of $x^2-a_px+p$; and $\nu \in H^1\bigl(\Q_p,\Z_p(\chi\psi^2)\bigr)$.  By \cite{Gross}, the assumption that $f_E$ has no companion form in
$S_{p-1}(N_E,\Fp)$ guarantees that $\nu$ is non-vanishing modulo $p$.  It follows that $\nu$ is surjective; it is also ramified since $\Zp(\chi\psi^2)$ has trivial inertia invariants.  Since $\chi$ is also ramified and surjective, this shows that the decomposition and inertia groups are as asserted.

\end{proof}

Passing from the factorization type to the coefficients $z_n$ of $\zeta_K(s)$ is straightforward.  
To compute $z_{p^n}$ for a prime power $p^n$ one need only know the factorization type 
$$a_1 \times (b_1,c_1) + \cdots + a_m \times (b_m,c_m)$$
of $p\O_K$.  We then know that there are precisely $a_i$ prime ideals of $\O_K$ of
absolute norm $p^\frac{b_i}{c_i}$.  The Euler factor at $p$ of $\zeta_K$ is thus
$$\prod_{i=1}^{m} \frac{1}{\left(1-p^{-\frac{b_i}{c_i}s}\right)^{a_i}}.$$
Multiplying out this finite product yields $z_{p^n}$ for all powers of $p$, and multiplicativity
$$z_{mn}=z_{m}z_{n} \text{~for~} (m,n)=1$$
yields the remainder of the coefficients.

\begin{example}  We return to the example of the introduction: let $E=X_0(11)$ and $N=63$. 
According to \cite{LMFDB}, $E$ has full Galois image modulo $N$ for all $N$ relatively prime to 
$110$.  One checks easily that $f_E$ has no companion form in
$S_2(11,\F_3)$ or $S_4(11,\F_5)$.  $E$ has split multiplicative reduction at $11$, with
Tate period
$$\vartheta = 268452333237063282944 \cdot 11^5 + O(11^{25}) \in \Q_{11}.$$
Since $(v_{11}(\vartheta),N)=1$, one sees immediately that all images of $\vartheta$ have maximal order modulo $9$ and $7$; thus the factorization of $11$ in $K$ is
\begin{multline*}
\dcpm_9(11,1,0,0) \otimes \dcpm_7(11,1,0,0) = 6 \times (6,1) + 12 \times (6,3) + \\ 36 \times (9,9) + 6 \times (42,7) + 12 \times (42,21) + 36 \times (63,63).
\end{multline*}
In particular, the Euler factor of $\zeta_K(s)$ at $11$ is
$$\frac{1}{(1-11^{-s})^{72}} \cdot \frac{1}{(1-11^{-2s})^{24}} \cdot \frac{1}{(1-11^{-6s})^{12}}.$$

$E$ has good ordinary reduction at both $3$ and $7$, resulting in double coset types 
$$\dcp_7(\Frob_{3}) \otimes \dcpo_9(2) = 18 \times (48,6) + 6 \times (432,9)$$
and
$$\dcp_9(\Frob_{7}) \otimes \dcpo_7(5) =  18 \times (6,6) + 18 \times (18,6) + 18 \times (42,7) + 18 \times (126,7).$$
The corresponding Euler factors are
$$\frac{1}{(1-3^{-8s})^{18}} \cdot \frac{1}{(1-3^{-48s})^{6}}$$
and
$$\frac{1}{(1-7^{-s})^{18}} \cdot \frac{1}{(1-7^{-3s})^{18}} \cdot 
\frac{1}{(1-7^{-6s})^{18}} \cdot \frac{1}{(1-7^{-18s})^{18}}.$$

All other Euler factors are easily computed using unramified double coset types.  For example,
$\Frob_{2} = \mat{-1}{1}{-1}{-1}$, so that $\dcp_{63}(\Frob_{2}) = 144 \times 24$, resulting in an
Euler factor of
$$\frac{1}{(1-2^{-24s})^{144}}.$$
The first nonzero coefficient of $\zeta_K(s)$ coming from an unramified prime is $c_{313}$:
$\Frob_{313} = \mat{0}{1}{-313}{-1}$ lies in
conjugacy class $$\text{II}(4,0)_9 \times \text{III}(1,5)_7,$$ yielding a double coset type of
$$\dcp_{63}(\Frob_{313}) = 36 \times 1 + 24 \times 3 + 126 \times 6 + 36 \times 9 + 126 \times 18.$$
\end{example}

\section{Elliptic Curves}

Let $E$ be an elliptic curve over $\Q$ and fix a prime $p$.  Since every element of $E(\Fpbar)$ is torsion and every characteristic $p$ torsion point of $E$ is defined over $\Fpbar$, we have
\begin{equation} \label{eq:silly}
E(\Fpbar) \cong \begin{cases} \Qp/\Zp \times \prod_{\ell \neq p} (\Ql/\Zl)^2 & E \text{~$p$-ordinary;} \\
\prod_{\ell \neq p} (\Ql/\Zl)^2 & E \text{~$p$-supersingular.} \end{cases}
\end{equation}
Consider now a second prime $q$.  Writing $A^{pq}$ for the prime-to-$pq$ part of an abelian group $A$, we see from (\ref{eq:silly}) that
$$E(\Fpbar)^{pq} \cong E(\Fqbar)^{pq}$$
as abelian groups. 

In fact, this isomorphism can be realized somewhat more conceptually.  Let $M$ be the extension of $\Q$ generated by all torsion points of $E$.  Fix embeddings
$\Qbar \inj \Qbar_p$ and $\Qbar \inj \Qbar_q$.
Since the corresponding reduction map
$$E(M)^{p}_{\tors} \to E(\Fpbar)^{p}$$
is injective, it is in fact an isomorphism.  Composing the inverse isomorphism with the reduction map
$$E(M) \to E(\Fqbar)$$
we obtain a homomorphism
$$h: E(\Fpbar)^p \to E(\Fqbar)$$
which is an isomorphism on prime-to-$pq$ parts.  

In this section we consider the following question: what are the possible $d$ such that
$h(E(\Fp)) \subseteq E(\Fqbar)$ is defined over $\F_{q^d}$?  (Obviously this depends on the choice of primes above
$p$ and $q$.)

Let us offer an approach to this question without reference to infinite degree extensions.  Write
$$E(\Fp) \cong \Z/N \times \Z/N'$$
with $N' \mid N$.  Assume that $p \nmid N$.  Set $L=\Q(E[N])$.  Fix a prime $\p$ of $L$ above $p$.  As above, the reduction map
$$E(L)_{\tors} \to E(\O_L/\p)$$
is injective.  Since by the choice of $N$ we know that $E(\Fp)$ must lie in the image, we may identify it with a subgroup $C_\p \subseteq E[N](L)$.  Let $\Gamma_\p \subseteq \Gal(L/\Q)$ denote the stabilizer of $C_\p$.  We define a field
$K$ as the fixed field of $L$ by $\Gamma_\p$.

Let $q$ be a second prime with prime factorization
$$q \O_K = \q_1 \cdots \q_r.$$
Then each composition
$$E(\Fp) \cong C_\p \subseteq E(K) \to E(\O_K/\q_i)$$
realizes $E(\Fp)$ inside of $E(\Fqbar)$, and our earlier question can be rephrased as asking what the possible residual degrees of the $\q_i$ are.  

Under the assumptions:
\begin{itemize}
\item $\Gal(L/\Q) \cong \GL_2(\Z/N)$;
\item $E(\Fp)$ cyclic (that is, $N'=1$);
\item $N$ odd;
\end{itemize}
this is precisely the question considered in this paper.  We can approach it either from the statistical point of view of Section~\ref{sec:dist} or the explicit point of view of 
Theorem~\ref{thm:main2}.  As usual, let us illustrate with an example.

\begin{example}
Consider the elliptic curve $X_0^+(37)$ with Weierstrass equation $$y^2+y=x^3-x$$ of conductor 37.  Take $p=4391$.  One computes that $E(\Fp) \cong \Z/4425$ where $4425 = 3 \cdot 5^2 \cdot 59$.  The extension $\Q(E[4425])^{\Gamma}$ has degree 16704000.  There are 668 distinct factorization types for this extension which we will not list here.   We instead focus on the minimal residual degree in each factorization, which range from $1$ to $3480$; see Table~\ref{t4425}.

\begin{table}
$$\begin{array}{|ll|ll|ll|} \hline
\text{Degree} & \text{Proportion} & \text{Degree} & \text{Proportion} & \text{Degree} & \text{Proportion} \\ \hline
1	&	0.04\%	&	24	&	0.06\%	&	232	&	2.58\%	\\
2	&	0.16\%	&	29	&	0.74\%	&	290	&	6.56\%	\\
3	&	0.00\%	&	30	&	0.21\%	&	348	&	0.49\%	\\
4	&	0.26\%	&	40	&	0.86\%	&	435	&	1.12\%	\\
5	&	0.13\%	&	58	&	2.03\%	&	580	&	8.28\%	\\
6	&	0.01\%	&	60	&	0.49\%	&	696	&	1.23\%	\\
8	&	0.17\%	&	87	&	0.04\%	&	870	&	3.80\%	\\
10	&	0.51\%	&	116	&	3.63\%	&	1160	&	15.31\%	\\
12	&	0.02\%	&	120	&	1.63\%	&	1740	&	10.28\%	\\
15	&	0.05\%	&	145	&	2.67\%	&	3480	&	35.97\%	\\
20	&	0.52\%	&	174	&	0.15\%	\\ \hline
\end{array}$$
\caption{Minimal residual degrees for $N=4425$} \label{t4425}
\end{table}

For a specific example, take $q=73$.  Then 
$$\Frob_{q} = \mat{0}{1}{-73}{-1}$$
with double coset type
\begin{multline*}
80 \times 29 + 2360 \times 58 + 80 \times 87 + 4800 \times 116 + \\ 2360 \times 174 + 4800 \times 348 + 6000 \times 580 + 6000 \times 1740.
\end{multline*}
That is, the least $d$ such that $E(\Fp) \subseteq E(\F_{q^d})$ is $d=29$, but much more likely is that $d$ could be as large as 1740.  

The least $q$ such that one can have $d=1$ is $q=4391$ (which not surprisingly is close to $4425$):
$$\Frob_{q,4425} = \mat{-16}{1}{-4119}{-17}$$
with double coset type
\begin{multline*}
2320 \times 1 + 3480 \times 2 + 1856 \times 5 + 2784 \times 10 + 2320 \times 25 + 4720 \times 29 + 3480 \times 50 + \\
7080 \times 58 + 3776 \times 145 + 5664 \times 290 + 4720 \times 725 + 7080 \times 1450.
\end{multline*}

\end{example}

\section{Software implementation}

Most of the results of this paper are very difficult to apply by hand.  For this reason we provide a collection of Magma programs which implement the results of this paper.  The programs are available for download.  Here we only discuss the highlights.

\begin{description}
\item[\texttt{StdDCP(p,n,o,a,b)}] Computes the double coset type $\dcp(\texttt{o};\texttt{a},\texttt{b})$ over $\Z/\texttt{p}^\texttt{n}$ as in Table~\ref{t:dcp}.
\item[\texttt{StdDCPP(p,n,ka,kb,a,b)}] Computes the double coset type $\dcp(\texttt{ka},\texttt{kb};\texttt{a},\texttt{b})$ over $\Z/\texttt{p}^\texttt{n}$ as in Table~\ref{t:dcp}.
\item[\texttt{DCP(p,n,g)}] Computes the double coset type of the matrix $\texttt{g} \in \GL_2(\Z/\texttt{p}^\texttt{n})$.
\item[\texttt{DCPN(N,g)}] Computes the double coset type of the matrix $\texttt{g} \in \GL_2(\Z/\texttt{N})$.
\item[\texttt{MultDCP(p,n,g,h)}] Computes the double coset type $\dcpm_{\texttt{p}^\texttt{n}}(\texttt{g}[1],\texttt{g}[4],\texttt{g}[2],\texttt{h}[2])$
\item[\texttt{OrdDCP(p,n,a)}] Computes the double coset type $\dcpo_{\texttt{p}^\texttt{n}}(\texttt{a})$
\item[\texttt{Frob(E,q)}] Computes the Duke--Toth Frobenius matrix for the prime $\texttt{q}$ acting on the elliptic curve $\texttt{E}$
\item[\texttt{CountIdealNorm(E,N,A,B)}] Computes the coefficients $c_{\texttt{A}},\ldots,c_{\texttt{B}}$ of the Dedekind zeta function of
the field $\Q(\texttt{E}[\texttt{N}])^{\Gamma}$ as in Theorem~\ref{thm:main2}.

\end{description}

\end{document}